\documentclass[12pt]{amsart}

\usepackage{amssymb,latexsym}
\usepackage[enableskew]{youngtab}
\usepackage{verbatim,euscript}
\usepackage{fullpage,color}

\usepackage[colorlinks=true,linkcolor=blue,urlcolor=violet,citecolor=magenta]{hyperref}

\numberwithin{equation}{section}

\input {cyracc.def}
\font\tencyr=wncyr10 
\font\tencyi=wncyi10 
\font\tencysc=wncysc10 
\def\rus{\tencyr\cyracc}
\def\rusi{\tencyi\cyracc}
\def\rusc{\tencysc\cyracc}

\newtheorem{thm}{Theorem}[section]
\newtheorem{lm}[thm]{Lemma}
\newtheorem{cl}[thm]{Corollary}
\newtheorem{prop}[thm]{Proposition}
\newtheorem{conj}[thm]{Conjecture}

\theoremstyle{remark}

\theoremstyle{definition}
\newtheorem{ex}[thm]{Example}
\newtheorem{df}{Definition}
\newtheorem{rmk}[thm]{Remark}

\definecolor{my_color}{rgb}{0,0.5,0.5}
\definecolor{MIXT}{rgb}{0.4,0.3,0.6}
\definecolor{blue-violet}{rgb}{0.54, 0.17, 0.89}
\definecolor{darkviolet}{rgb}{0.58, 0.0, 0.83}
\definecolor{violet-ryb}{rgb}{0.53, 0.0, 0.69}
\definecolor{violet}{rgb}{0.62, 0.0, 1.0}

\newcommand {\be}{{\mathfrak b}}
\newcommand {\ce}{{\mathfrak c}}

\newcommand {\g}{{\mathfrak g}}
\newcommand {\h}{{\mathfrak h}}

\newcommand {\el}{{\mathfrak l}}

\newcommand {\fN}{{\mathfrak N}}
\newcommand {\p}{{\mathfrak p}}

\newcommand {\te}{{\mathfrak t}}
\newcommand {\ut}{{\mathfrak u}}

\newcommand {\gln}{{\mathfrak{gl}}_n}
\newcommand {\sln}{{\mathfrak{sl}}_n}
\newcommand {\sltn}{{\mathfrak{sl}}_{2n}}

\newcommand {\spn}{{\mathfrak{sp}}_{2n}}
\newcommand {\son}{{\mathfrak{so}}_{n}}

\newcommand {\sone}{{\mathfrak{so}}_{2n}}

\newcommand {\esi}{\varepsilon}
\newcommand {\ap}{\alpha}

\newcommand {\lb}{\lambda}

\newcommand {\cI}{{\mathcal I}}

\newcommand {\co}{{\mathcal O}}

\newcommand {\BC}{{\mathbb C}}
\newcommand {\BZ}{{\mathbb Z}}

\newcommand {\ad}{{\mathrm{ad\,}}}

\newcommand {\codim}{{\mathrm{codim}}}

\newcommand {\Ind}{{\mathsf{Ind}}}
\newcommand {\Lie}{{\mathrm{Lie\,}}}
\newcommand {\rk}{{\mathsf{rk\,}}}

\newcommand {\tri}{\mathfrak{sl}_2}
\newcommand {\GR}[2]{{\textrm{{\sf\bfseries #1}}}_{#2}}
\newcommand {\GRt}[2]{{\tilde{\textrm{{\sf\bfseries #1}}}}_{#2}}
\newcommand {\ov}{\overline}
\newcommand {\un}{\underline}

\newcommand {\beq}{\begin{equation}}
\newcommand {\eeq}{\end{equation}}

\newcommand{\curge}{\succcurlyeq}
\newcommand{\curle}{\preccurlyeq}
\renewcommand{\le}{\leqslant}
\renewcommand{\ge}{\geqslant}

\newcommand{\ph}{\p_{\langle h\rangle}}
\newcommand{\eus}{\EuScript}
\newcommand{\blb}{\boldsymbol{\lb}}

\newenvironment{Dn}[6]{%
{\small\begin{tabular}{@{}c@{}}
{#1}--$\,\cdots$ --{#2}--{#3}--\lower3.5ex\vbox{\hbox{{#4}\rule{0ex}{2.5ex}}
\hbox{\hspace{0.4ex}\rule{.1ex}{1ex}\rule{0ex}{1.4ex}}\hbox{{#6}\strut}}--{#5}
\end{tabular}}}

\newenvironment{E6}[6]{%
{\small\begin{tabular}{@{}c@{}}
{#1}--{#2}--\lower3.5ex\vbox{\hbox{{#3}\rule{0ex}{2.5ex}}
\hbox{\hspace{0.4ex}\rule{.1ex}{1ex}\rule{0ex}{1.4ex}}\hbox{{#6}\strut}}--{#4}--{#5}
\end{tabular}}}

\newenvironment{E7}[7]{%
{\small\begin{tabular}{@{}c@{}}
{#1}--{#2}--{#3}--\lower3.5ex\vbox{\hbox{{#4}\rule{0ex}{2.5ex}}
\hbox{\hspace{0.4ex}\rule{.1ex}{1ex}\rule{0ex}{1.4ex}}\hbox{{#7}\strut}}--{#5}--{#6}
\end{tabular}}}

\newenvironment{E8}[8]{%
{\small\begin{tabular}{@{}c@{}}
{#1}--{#2}--{#3}--{#4}--\lower3.5ex\vbox{\hbox{{#5}\rule{0ex}{2.5ex}}
\hbox{\hspace{0.4ex}\rule{.1ex}{1ex}\rule{0ex}{1.4ex}}\hbox{{#8}\strut}}--{#6}--{#7}
\end{tabular}}}

\begin{document}
\setlength{\parskip}{2pt plus 4pt minus 0pt}
\hfill {\scriptsize February 26, 2018} 
\vskip1.5ex

\title{Nilpotent subspaces and nilpotent orbits}
\author[D.\,Panyushev]{Dmitri I. Panyushev}
\address[D.P.]
{Institute for Information Transmission Problems of the R.A.S, Bolshoi Karetnyi per. 19, 
127051 Moscow,   Russia}
\email{panyushev@iitp.ru}
\author[O.\,Yakimova]{Oksana S. Yakimova}
\address[O.Y.]{Institut f\"ur Mathematik, Friedrich-Schiller-Universit\"at Jena,  07737 Jena, 
Deutschland}
\email{oksana.yakimova@uni-jena.de}
\thanks{The research of the first author was carried out at the IITP RAS at the expense of the Russian Foundation for Sciences (project {\rus N0} 14-50-00150). 
The second author is partially supported by the DFG priority programme SPP 1388 
``Darstellungstheorie" and by the Graduiertenkolleg GRK 1523 ``Quanten- und Gravitationsfelder".}

\begin{abstract}
Let $G$ be a semisimple complex algebraic group with Lie algebra $\g$. For a nilpotent $G$-orbit $\co\subset\g$, 
let $d_\co$ denote the maximal dimension of a subspace $V\subset \g$ that is contained in the closure of 
$\co$. In this note, we prove that $d_\co \le \frac{1}{2}\dim\co$ and this upper bound is attained if and 
only if $\co$ is a Richardson orbit. Furthermore, if $V$ is $B$-stable and 
$\dim V= \frac{1}{2}\dim\co$, then $V$ is the nilradical of a polarisation of $\co$. Every nilpotent orbit 
closure has a distinguished $B$-stable subspace constructed via an $\tri$-triple, 
which is called the {\it Dynkin ideal}. We then characterise the 
nilpotent orbits $\co$ such that the Dynkin ideal {\bf (1)} has the minimal 
dimension among all $B$-stable subspaces $\ce$ such that $\ce\cap\co$ is dense in $\ce$, or {\bf (2)} 
is the only  $B$-stable subspace $\ce$ such that $\ce\cap\co$ is dense in $\ce$.
\end{abstract}
\maketitle

\section*{Introduction}

\noindent
Let $G$ be a connected complex semisimple algebraic group with Lie algebra $\g$. 
We fix a Borel subgroup $B\subset G$, a maximal torus $T\subset B$, and the corresponding 
triangular decomposition $\g=\ut^-\oplus\te\oplus\ut$. Here $\Lie(B)=\be=\te\oplus\ut$.
Let $\fN$ denote the cone of nilpotent elements of $\g$ and $\fN/G$ the (finite) set of $G$-orbits
in $\fN$. A subspace $V$ of $\g$ is said to be {\it nilpotent\/} if
$V\subset \fN$.  Then $\ov{G{\cdot}V}$ is the closure of a nilpotent orbit. If $\ov{G{\cdot}V}=\ov{\co}$ for 
some $\co\in\fN/G$, then we say that $V$ is {\it associated\/} with $\co$.
In general, one has to distinguish the numbers 
$d_\co=\max\{\dim V\mid V\subset \ov{\co}\}$ and 
$\bar d_\co=\max\{\dim V\mid V\subset \ov{\co} \ \& \ V\cap\co\ne \varnothing\}=
\max\{\dim V\mid \ov{G{\cdot}V}=\ov{\co}\}$. Clearly, $\bar d_\co\le d_\co$.

In this article, we study the set of nilpotent subspaces associated with a given $\co$. Our 
main observation is that if $V\subset\ov{\co}$, then $\dim V\le \frac{1}{2}\dim\co$ 
and this upper bound is attained if and only if $\co$ is a Richardson orbit. 
Furthermore, if $V$ is $B$-stable and $\dim V= \frac{1}{2}\dim\co$, then $V$ is the nilradical of a 
polarisation of $\co$ (Theorem~\ref{thm:main-half}). In other words, $d_\co\le \frac{1}{2}\dim\co$ and 
$\co\subset \fN$ is Richardson  if and only if\/ $\bar d_\co=d_\co=\frac{1}{2}\dim\co$.
Our upper bound for $\dim V$ generalises a classical result of Gerstenhaber for $\sln$ \cite{nil-4}.
Recall that  $\co\in\fN/G$ is said to be {\it Richardson},  
if there is a parabolic subalgebra $\p$, with nilradical $\p^{nil}$, such that $\co\cap\p^{nil}$ is dense in 
$\p^{nil}$.  Then $\p$ is called a {\it polarisation\/} of $\co$ and $\dim\co=2\dim\p^{nil}$. Without loss of 
generality, one may assume that a polarisation $\p$ is {\it standard\/} (i.e., $\p\supset\be$) and then $\p^{nil}\subset \ut$ 
is $B$-stable. 

This result suggests that it is interesting to study $B$-stable subspaces associated with $\co$.
Let $\cI(\be)$ denote the finite set of all $B$-stable subspaces (=\,Ó$\be$-ideals) of $\ut$. 
If $V$ is nilpotent and $B$-stable, then $V\subset\ut$, i.e., $V\in\cI(\be)$. Restricting ourselves with the $B$-stable 
subspaces associated with $\co$, we obtain a subset $\cI(\be)_\co$ and the partition
$ \cI(\be)=\bigsqcup_{\co\in\fN/G}\cI(\be)_\co$. We regard $\cI(\be)$ as a poset with respect to inclusion, and then 
each $\cI(\be)_\co$ is a subposet of $\cI(\be)$. The sets $\cI(\be)_\co$ were first considered by 
Sommers~\cite{som}, who called them "equivalence classes of ideals".  Using a suitable $\tri$-triple and 
$\BZ$-grading of $\g$ associated with $\co$, one defines a special element of $\cI(\be)_\co$, which is
called the {\it Dynkin ideal}, denoted $\ce_{\sf Dy}(\co)$.  We consider the following numbers related to
$\cI(\be)_\co$:
\begin{itemize}
\item $d_{\sf min}(\co)=\min\{\dim\ce\mid \ce\in \cI(\be)_{\co}\}$;
\item $d_{\sf Dy}(\co)=\dim \ce_{\sf Dy}(\co)$, the dimension of the Dynkin ideal. 
\item $d_{\sf max}(\co)=\max\{\dim\ce\mid \ce\in \cI(\be)_{\co}\}$;
\end{itemize}
Then $d_{\sf min}(\co)\le d_{\sf Dy}(\co)\le d_{\sf max}(\co)\le \frac{1}{2}\dim\co$ and 
$d_{\sf max}(\co)= \frac{1}{2}\dim\co$ if and only if $\co$ is Richardson. 
It is also interesting to study maximal and minimal elements of the poset $\cI(\be)_\co$, and
we notice that if $\co$ is induced 
from a nilpotent orbit $\co'$ in a Levi subalgebra $\el$, then the maximal elements of
$\cI(\be_\el)_{\co'}$ yield maximal elements of $\cI(\be)_\co$ (Lemma~\ref{lm:induction-poset}).

We say that $\co$ is {\it extreme}, if $d_{\sf min}(\co)= d_{\sf Dy}(\co)$. Using a general formula for
$d_{\sf min}(\co)$ suggested by Sommers, see Eq.~\eqref{eq:d-min}, and verified case-be-case in \cite{kaw86, som,fang14}, 
we provide two characterisations of the extreme 
orbits~(Prop.~\ref{prop:extrim}) and classify them in all simple Lie algebras. 
It is easily seen that the {\it principal\/} nilpotent orbit $\co_{pr}$ and the {\it minimal\/} 
nilpotent orbit $\co_{min}$ are always extreme.
An intriguing {\it a posteriori\/} fact is that if the highest root of $\g$ is fundamental, then
there are exactly three extreme orbits: $\co_{pr}$, $\co_{min}$, and yet another one, which is said to 
be {\it intermediate} and denoted $\co_{imd}$ (Theorem~\ref{thm:extrem-excepts}). Moreover, the intermediate orbits admit a uniform 
description via the principal nilpotent orbit in the Levi subalgebra associated with highest root 
(Prop.~\ref{prop:extrim-uniform}). The highest root is {\bf not} fundamental if and only if $\g$ is of type 
$\GR{A}{n}$ or $\GR{C}{n}$. The number of nonzero extreme orbits in $\spn$ is $n$; and for
$\sln$, this number is at least $[n/2]$, depending on the parity of $n$, see Theorem~\ref{thm:extrem-sl-sp}. 

We say that $\co$ is {\it lonely}, if $\ce_{\sf Dy}(\co)$ is the only element of $\cI(\be)_\co$.
Obviously, a lonely orbit is extreme, and using our classification of extreme orbits, we classify the lonely 
orbits in Section~\ref{sect:lonely}. It is easily seen that $\ut$ is the only $\be$-ideal associated with 
$\co_{pr}$; hence $\co_{pr}$ is always lonely. A complete description is the following, see 
Theorem~\ref{thm:all-lonely}: 
\\  \indent
(1) \ for $\GR{A}{n}$ ($n\ge 1$), $\GR{B}{n}$ ($n\ge 3$), and $\GR{F}{4}$, the only lonely orbit is $\co_{pr}$;
\\  \indent
(2) \ for $\GR{D}{n}$ ($n\ge 4$), $\GR{G}{2}$, and $\GR{E}{n}$ ($n=6,7,8$), the lonely orbits are $\co_{pr}$
and $\co_{imd}$; 
\\  \indent
(3) \ for $\GR{C}{n}$ ($n\ge 1$), all extreme orbits are lonely.
\\
In particular, $\co_{min}$ is lonely only  for the symplectic Lie algebras.
Recall that $\tri=\mathfrak{sp}_2$ and here $\co_{pr}=\co_{min}$.

{\bf Notation}.  Associated with our fixed triangular decomposition, there are the following objects:
$W$ is the Weyl group and $\Delta$ is the roots system of $(\g,\te)$. Then $\Delta^+\subset\Delta$ 
corresponds to $\ut$, $\Pi$ is the set of simple roots in $\Delta^+$, and $\g^\gamma$ is the root 
space corresponding to $\gamma\in\Delta$. 
If $\g$ is simple, then $\theta\in\Delta^+$ is the highest root. 
If $H\subset G$ and $e\in\g$, then $H_e$ is the centraliser of $e$ in $H$ and $\h_e$ is the centraliser of
$e$ in $\h=\Lie H$.

For any $\eus S\subset\Pi$, $\p\{\eus S\}$ stands for the standard parabolic subalgebra such that 
$\eus S$ is exactly the set of simple roots occurring in $\p\{\eus S\}^{nil}$. Then $\el\{\eus S\}$ is the 
{\it standard Levi subalgebra\/} of $\p\{\eus S\}$ (i.e., $\el\{\eus S\}\supset \te$ and $\Pi\setminus\eus S$ is 
the set of simple roots of $\el\{\eus S\}$). In particular, $\p\{\varnothing\}=\g$ and $\p\{\Pi\}=\be$.

Algebraic groups are denoted by capital Roman letters, and their Lie algebras are denoted by the corresponding small Gothic letters (and vice versa).
Our basic references for semisimple Lie algebras and their nilpotent orbits are \cite{cm, t41}.

\section{Preliminaries on nilpotent orbits}
\label{sect:prelim}

\noindent
\subsection{$\tri$-triples, gradings, and centralisers.}   \label{subs:sl2}
Recall the Dynkin-Kostant theory on $\tri$-triples and nilpotent orbits, 
see~\cite[Chapter\,3]{cm},\,\cite[Chapter\,6, \S\,2]{t41}.
For $e\in \fN\setminus\{0\}$, let $\{e,h,f\}$ be an $\tri$-triple (that is, $[h,e]=2e$, $[e,f]=h$, and $[h,f]=-2f$). 
The semisimple element $h$ determines the
$\BZ$-grading of   $\g=\bigoplus_{i\in\BZ} \g(h;i)$, where $\g(h;i)$ is the $i$-eigenspace of $\ad h$ in
$\g$  (hence $e\in \g(h;2)$). Set $\g(h;{\ge}k)=\bigoplus_{i\ge k} \g(h;i)$.
Replacing $\{e,h,f\}$ with a $G$-conjugate $\tri$-triple, we may assume that $h$ is {\it dominant}, i.e.,
$h\in\te$ and $\ap(h)\ge 0$ for all $\ap\in\Pi$. Such a triple is said to be {\it adapted\/} (to the chosen 
triangular decomposition of $\g$). Then $\ap(h)\in\{0,1,2\}$ and the {\it weighted Dynkin diagram\/} 
(=\,{\sf wDd}) is obtained by putting the integers $\ap(h)$ ($\ap\in\Pi$) at the corresponding nodes of the 
Dynkin diagram. 
\\ \indent 
If $h$ is dominant, then the subspaces $\g(h;{\ge}k)$ are $B$-stable. Moreover,
$\ph:=\g(h;{\ge }0)$ is a standard parabolic subalgebra,  $\g(h;0)$ 
is a standard Levi subalgebra of $\ph$, and $\ph^{nil}=\g(h;{\ge}1)\subset \ut$. Furthermore, 
$[\g(h;0),e]=\g(h;2)$ and $[\ph,e]=\g(h;{\ge}2)$. In terms of the corresponding connected groups
$G(h;0)$ and $P_{\langle h\rangle}$, this means that 
$G(h;0){\cdot}e$ is dense in $\g(h;2)$ and $P_{\langle h\rangle}{\cdot}e$ is dense in $\g(h;{\ge}2)$.
Note also that $\ph=\p\{\eus S_h\}$, where $\eus S_h=\{\ap\in \Pi\mid \ap(h)>0\}$.
For an adapted $\tri$-triple, we have $\be=\be(h;0)\oplus \g(h;{\ge}1)$, where
$\be(h;0)=\be\cap\g(h;0)$ is a Borel subalgebra of $\g(h;0)$.

An element $e\in \fN$ (or orbit $G{\cdot}e$) is {\it even}, if the eigenvalues of $\ad h$ are even. It is 
equivalent to that  $\g(h;1)=0$ or $[\ph,e]=\ph^{nil}$. 
Since $\dim\g_e=\dim\g(h;0)+\dim\g(h;1)$ and $\dim\g(h;i)=\dim\g(h;-i)$, we have
$\dim\g(h;{\ge}2)=(\dim\co-\dim\g(h;1))/2$. Therefore, $\co$ is even if and only if 
$\dim\g(h;{\ge}2)=\frac{1}{2}\dim\co$.

\subsection{Induced and Richardson orbits.}   \label{subs:induced}
Let $\co=G{\cdot}e$ be a nilpotent orbit in $\g$. Let $\p$ be a parabolic subalgebra of $\g$, $\el$ a Levi 
subalgebra of $\p$ and $\co'=L{\cdot}e'$ a nilpotent $L$-orbit in $[\el,\el]$. Then $\co$ is said to be {\it induced from}\/ $(\co', \el,\p)$, if
$\co$ is dense in $G{\cdot}(e'+\p^{nil})$. 
Basic results on the induction of nilpotent orbits can be found in~\cite[Chapter\,7]{cm}.
This construction depends only on (the conjugacy class of) $\el$. That is, 
if $\tilde\p$ is another parabolic subalgebra with the same Levi subalgebra $\el$, then
the dense $G$-orbits in $G{\cdot}(e'+\p^{nil})$ and $G{\cdot}(e'+\tilde\p^{nil})$ coincide. Therefore, one can omit $\p$ from notation and denote the induced orbit by $\co=\Ind_\el^\g(\co')$.
Furthermore,
\beq    \label{eq:dim-induced}
  \dim\co=\dim\co'+ 2\dim\p^{nil}=\dim\co'+(\dim\g-\dim\el) .
\eeq
Equivalently,  $\codim_\g(\co)=\codim_\el(\co')$.
A nilpotent orbit $\co$ is  {\it Richardson} if it is induced from the zero orbit in 
some Levi subalgebra (that is, $e'=0$ in the above construction). Any of the corresponding parabolic 
subalgebras $\p$ is called a {\it polarisation\/} of $\co$. Then $\dim\co=2\dim\p^{nil}$ and $P{\cdot}e$ is 
dense in $\p^{nil}$. Without loss of generality, we may always assume that $\p$ is standard and hence 
$\p^{nil}\subset\ut$. An even nilpotent orbit is Richardson, because one can take $\ph=\g(h;{\ge}0)$ as 
a polarisation. 

If $\co$ cannot be induced from a proper Levi subalgebra $\el\subset\g$ and some $\co'\subset\el$, then 
$\co$ is said to be {\it rigid}. All Richardson  and rigid orbits are known, 
and we will freely use that information in what follows. Namely, Kempken explicitly described induction 
in the classical Lie algebras in terms of partitions~\cite{ke83}, see also~\cite[Ch.\,7]{cm}. For the exceptional Lie algebras, the induction was investigated by Elashvili. An account of his results is given in
\cite[pp.\,171--177]{spalt82}. 

\subsection{Orbital varieties.}
For $\co\in\fN/G$, the irreducible components of $\co\cap\ut$ are called the {\it orbital varieties} of $\co$. It is known that $\co\cap\ut$ is of pure dimension $\frac{1}{2}\dim\co$~\cite{spalt77,St76}.
If $\co$ is Richardson  and $\p$ is a standard polarisation of it, then $\p^{nil} \subset \ov{\ut\cap\co}$ and
$\dim\p^{nil}=\frac{1}{2}\dim\co$ in view of~\eqref{eq:dim-induced}. 
Therefore, $\p^{nil}$ is the closure of an orbital variety of $\co$.
In particular, for the Richardson  orbits, the closure of an orbital variety is smooth.
We will prove below that the converse is also true.

\subsection{The closure relation.}
Letting $\co_1\leq \co_2$ if  $\co_1\subset\ov{\co_2}$, we make $\fN/G$ a finite poset.

For the classical Lie algebras, the description of $(\fN/G, \leq)$ via partitions is due to Gerstenhaber 
and Hesselink, see~\cite[6.2]{cm}. Explicit results for the exceptional types are due to Shoji and Mizuno.
The corresponding Hasse diagrams are depicted in \cite[pp.247--250]{spalt82}.

\section{An upper bound}         
\label{sect:main-bound}

\noindent
In this section, we obtain an upper bound on the dimension of a linear subspace sitting in the closure of
a nilpotent orbit.
We begin with a preliminary result.

\begin{lm}    \label{lm:orb-var}
If $\co\ne \{0\}$, then the closure of any orbital variety of $\co$ contains a line of the form 
$\g^\ap$, $\ap\in\Pi$. In particular, an orbital variety of a nontrivial orbit does not belong to $[\ut,\ut]$.
\end{lm}
\begin{proof}
For $w\in W$, let $^w\ut$ denote the $w$-conjugate of $\ut$. That is,
$^w\ut=\oplus_{\gamma:\ w^{-1}\gamma\in\Delta^+}\g^\gamma$.
By \cite[9.6]{jos84}, the closure of an orbital variety is of 
the form $\ov{B{\cdot}(\ut\cap\, ^w\ut)}$ for some $w\in W$. 
The subspace $\ut\cap ^w\!\ut$ is nonzero if and only if $w$ is {\bf not} the longest element of $W$, and in that case $\ut\cap ^w\!\ut$ contains a root subspace $\g^\ap$ for some $\ap\in\Pi$.
\end{proof}

\begin{thm}   \label{thm:main-half}
Let $\co=G{\cdot}e\subset\g$ be a nilpotent orbit and $V$ a subspace in $\ov{\co}$. 
\begin{itemize}
\item[\sf (i)] \ Then  $\dim V\le \frac{1}{2}\dim\co$ and there is also a $B$-stable subspace in $\ov{\co}$
of the same dimension.
If\/ $V$ is $B$-stable, then $V\subset\ut$.
\item[\sf (ii)] \ 
If\/ $\dim V= \frac{1}{2}\dim\co$, then
$\co$ is Richardson;
furthermore, if\/ $V$ is $B$-stable and
$\dim V= \frac{1}{2}\dim\co$, then $V=\p^{nil}$, where $\p$
is a standard polarisation of $\co$. 
\end{itemize}
\end{thm}
\begin{proof}
(i) \  Consider $V$ as a point in the Grassmannian ${\sf Gr}_d(\g)$ of 
$d$-dimensional subspaces of $\g$, where $d=\dim V$. By the Borel fixed point theorem, the closure of 
$B{\cdot}\{V\}$ in  ${\sf Gr}_d(\g)$ contains a $B$-fixed point $\{V'\}$. Then $V'$ is a $d$-dimensional 
$B$-stable subspace sitting in $\ov{\co}$. In particular, $V'$ is $T$-stable and nilpotent. Therefore,
$V'=\oplus_{\gamma\in M}\g^\gamma$ for some $M\subset\Delta$. 
Assume that $\gamma\in M\cap\Delta^-$. Then $e_{-\gamma}\in\be$, 
$e_{-\gamma}(V')\subset V'$, and $[e_{-\gamma}, e_\gamma]$ is a nonzero semisimple element in
$V'$. Hence $M\subset \Delta^+$ and $V'\subset \ut\cap \ov{\co}$.  
Since $\dim (\ut\cap\co)=\frac{1}{2}\dim\co$ for any nilpotent orbit~\cite{spalt77,St76}, we conclude that
$\dim V=\dim V'\le \frac{1}{2}\dim\co$.

(ii) By part (i), we may assume that
$V$ itself is $B$-stable, and thereby $V\subset\ut\cap \ov{\co}$. If $\dim V= \frac{1}{2}\dim\co$, then
$V$ is the closure of an orbital variety of $\co$. Hence $V\not\subset [\ut,\ut]$ in view of 
Lemma~\ref{lm:orb-var}.  Set $\Pi_V=\{\ap\in\Pi\mid \g^\ap\subset V \}$, and let $\p=\p\{\Pi_V\}$ be the 
corresponding standard parabolic subalgebra. 
Recall that $\Pi\setminus \Pi_V$ is the set of simple roots of standard Levi subalgebra 
$\el\{\Pi_V\}=:\el\subset\p$. Let $L$ be the Levi subgroup corresponding to $\el$.
Since $V$ is $B$-stable, we have $V\supset \p^{nil}$. Hence $V=(V\cap\el)\oplus \p^{nil}$ and
$V_\el=V\cap \el$ is a $B\cap L$-stable subspace of $\ut_\el:=\ut\cap\el$.

Let $\co_\el$ be the dense $L$-orbit in $L{\cdot}V_\el$. Then $V_\el\subset \ov{\co}$ and the 
above description shows that $\co=\Ind_\el^\g(\co_\el)$. Hence
\[
   2\dim V=\dim\co=\dim\co_\el+2\dim\p^{nil} .
\]
Since $\dim\co_\el=2(\dim V-\dim\p^{nil})=2\dim V_\el$, 
we see that $V_\el$ is the closure of an orbital variety of $\co_\el$. By the very construction of 
$\Pi_V$, $V_\el$ does not contain simple roots of $\el$, hence $V_\el\subset [\ut_\el,\ut_\el]$. 
It then follows from Lemma~\ref{lm:orb-var} that $\co_\el=\{0\}$ and hence $V_\el=\{0\}$. 
Thus, $V=\p^{nil}$, and we are done.
\end{proof}

We provide below some nice-looking consequences of this theorem.
\begin{cl}
If\/ $V\subset\fN$ is $B$-stable, then $\dim G{\cdot}V\ge 2\dim V$. Moreover, the equality holds if and only 
if\/ $V=\p^{nil}$ for a standard parabolic subalgebra $\p$.
\end{cl}

\begin{cl}    \label{prop:equiv-glatt}
For a nilpotent otbit $\co$, the following properties are equivalent:
\begin{itemize}
\item[\sf (i)] \ $\co$ is Richardson;
\item[\sf (ii)] \  there is an orbital variety of $\co$ whose closure  is smooth;
\item[\sf (iii)] \ there is an orbital variety of $\co$ whose closure is an affine space.
\end{itemize}
\end{cl}
\begin{proof}
(i)$\Rightarrow$(ii) \ We already noticed that if $\p$ is a standard polarisation of $\co$, then 
$\p^{nil}$ is the closure of an orbital variety of $\co$.

(ii)$\Rightarrow$(iii) \ Since $\co$ is a cone, the closure of any orbital variety $\eus X$ is an affine 
conical variety with vertex $\{0\}$ (i.e., if $x\in\eus X$ and $t\in\BC$, then $tx\in\eus X$).  As is 
well known (and easily seen), an affine conical smooth variety is necessarily an affine space.

(iii)$\Rightarrow$(i) \ This is proved in Theorem~\ref{thm:main-half}(ii).
\end{proof}
Set $\bar d_\co=\max\{\dim V\mid V\subset \ov{\co} \ \& \ V\cap\co\ne \varnothing\}=
\max\{\dim V\mid \ov{G{\cdot}V}=\ov{\co}\}$  and \\ $d_\co=\max\{\dim V\mid V\subset \ov{\co}\}$. 
Then
$d_\co=\displaystyle\max_{\co'\subset\ov{\co}}\bar d_{\co'}$, \ 
$d_\co \le \frac{1}{2}\dim\co$, \ and an obvious consequence of Theorem~\ref{thm:main-half} is

\begin{cl}
The orbit $\co\subset \fN$ is Richardson  if and only if\/ $\bar d_\co=d_\co=\frac{1}{2}\dim\co$.
\end{cl}

For the non-Richardson  orbits, we have $\bar d_\co \le d_\co < \frac{1}{2}\dim\co$, and it is not 
known to the authors what is the precise value of $d_\co$ and
whether it is always true that $\bar d_\co = d_\co$.

\begin{ex}  \label{ex:left-factor}
(a) \ For $\g=\sln$, the nilpotent orbits are parametrised by the partitions of $n$.
Let $\blb=(\lb_1,\dots,\lb_t)$ be such a partition and $\co(\blb)$ the corresponding orbit. That is, 
$\lb_1,\dots,\lb_t$ are the sizes of blocks in the Jordan normal form for $e\in \co(\blb)$.
Let $\hat\blb=(c_1,\dots,c_m)$ be the dual partition.
By \cite[Theorem\,2]{nil-4}, if $V\subset \ov{\co(\blb)}$, then 
\[
   \dim V\le \frac{1}{2}(n^2-\sum_{j=1}^m c_j^2)=\frac{1}{2}\dim \co(\blb) \ . 
\]
It is also shown in~\cite{nil-4} that this upper bound on $\dim V$ is achieved for 
any $\blb$. Since all nilpotent orbits in $\sln$ are Richardson ~\cite[7.2]{cm}, 
our Theorem~\ref{thm:main-half} is a 
generalisation of that classical result of Gerstenhaber to arbitrary semisimple Lie algebras.

(b) \ For $\g=\spn$ (resp. $\g=\son$), the nilpotent orbits correspond to the partitions $\blb$ of $2n$ 
(resp. $n$) such that each odd (resp. even) part occurs an even number of times~\cite[5.1]{cm}. 
(Recall also that if $\blb$ is
{\it very even} in the orthogonal case, i.e., all $\lb_i$ are even, then there are two different corresponding $SO_n$-orbits.) 
In both cases, not all nilpotent orbits are Richardson. Therefore, the dimension formula for $\co(\blb)$~\cite[Cor.\,6.1.4]{cm} does not provide a precise upper bound on $\dim V$ for all orbits.
\end{ex}

\section{Posets of $B$-stable nilpotent subspaces}
\label{sect:equiv-classes}

\noindent 
Let $\cI(\be)$ denote the set of all $B$-stable subspaces (=$\be$-ideals) of $\ut$ (also called 
{\it ad-nilpotent ideals}). It is a finite poset with respect to 
inclusion, and there is a rich combinatorial theory related to $\cI(\be)$ that we 
do not touch upon here, see~\cite{cp1, cp2}. Any $\ce\in\cI(\be)$ is a sum of root spaces, and we write $\Delta(\ce)$ for the respective set
of positive roots. That is, $\gamma\in\Delta(\ce)$ if and only if $\g^\gamma\subset\ce$.
A usual partial order `$\curle$' in $\Delta^+$ is defined by the condition that $\gamma\curle \mu$
if and only if $\mu-\gamma$ is a nonnegative linear combination of simple roots.
Then $\Delta(\ce)$ is an upper ideal of $(\Delta^+, \curle)$, and therefore it is fully determined by its 
subset of minimal elements with respect to  `$\curle$'.

If $\ce\in \cI(\be)$, then $G{\cdot}\ce$ is closed and thereby is the closure of a nilpotent orbit. The 
$\be$-ideals $\ce_1,\ce_2$ are {\it equivalent\/} if $G{\cdot}\ce_1=G{\cdot}\ce_2$. The equivalence classes are parametrised by the nilpotent 
orbits, and $\cI(\be)_{\co}$ stands for the class associated with $\co\subset\fN$. That is,
$\cI(\be)_{\co}=\{\ce\in \cI(\be)\mid G{\cdot}\ce=\ov{\co}\}$. 
Clearly, $\cI(\be)_{\co}$ is a subposet of $\cI(\be)$ and
by Theorem~\ref{thm:main-half}(i), it is also the set of {\bf all} $B$-stable subspaces in $\ov{\co}$ that 
meet $\co$. The classes $\cI(\be)_{\co}$ have first been considered by Sommers~\cite{som}. 
For $\ce\in \cI(\be)_\co$, we also say that (the $\be$-ideal) $\ce$ {\it is associated\/} with $\co$.

If $\{e,h,f\}$ is an adapted $\tri$-triple ($e\in\co$), then the $\be$-ideal $\g(h;{\ge}2)$
belongs to $\cI(\be)_{\co}$, see~\ref{subs:sl2}. Hence  $\cI(\be)_{\co}$ is nonempty for any $\co$. 
Following~\cite{som}, we say that $\ce_{\sf Dy}(\co):=\g(h;{\ge}2)$ 
is the {\it Dynkin ideal\/} associated with $\co=G{\cdot}e$ 
(and an adapted $\tri$-triple $\{e,h,f\}$).

Given $\co\in\fN/G$, we consider the following numbers related to $\cI(\be)_\co$:
\begin{itemize}
\item $d_{\sf min}(\co)=\min\{\dim\ce\mid \ce\in \cI(\be)_{\co}\}$;
\item $d_{\sf Dy}(\co)=\dim \ce_{\sf Dy}(\co)$, the dimension of the Dynkin ideal;
\item $d_{\sf max}(\co)=\max\{\dim\ce\mid \ce\in \cI(\be)_{\co}\}$.
\end{itemize}
Note that the definition of  $d_\co$ and $\bar d_\co$ in Section~\ref{sect:main-bound} 
exploits arbitrary nilpotent subspaces $V$, whereas here we only allow $B$-stable subspaces. 
Obviously, $d_{\sf max}(\co)\le \bar{d}_\co$. Hence
\beq    \label{eq:mnogo-neravenstv}
   d_{\sf min}(\co) \le d_{\sf Dy}(\co) \le d_{\sf max}(\co) \le \bar{d}_\co \le d_\co \le \frac{1}{2}\dim\co .
\eeq
It follows from Theorem~\ref{thm:main-half} that $d_{\sf max}(\co)=\frac{1}{2}\dim\co$  if and only if
$\co$ is Richardson. Furthermore, as noted in Section~\ref{subs:sl2}, $d_{\sf Dy}(\co)=\frac{1}{2}\dim\co$  
if and only if $\co$ is even. As we shall see in Section~\ref{sect:anomalies}, it can happen that
$\co_1\subset \ov{\co_2}$, but $d_{\sf max}(\co_1)> d_{\sf max}(\co_2)$. The reason is that if 
$V\subset \ov{\co}$, $V\cap \co\ne \varnothing$, and $V'$ is a $B$-stable subspace  in the 
closure of $B{\cdot}\{V\}$ in the Grassmannian ${\sf Gr}_{\dim V}(\g)$ (see proof of Theorem~\ref{thm:main-half}), then it is possible that
$V'\cap \co=\varnothing$.
That is, it can happen that $d_{\sf max}(\co)$ is considerably less that $d_\co$.
However, the following is true.

\begin{prop}    \label{d_O-max}
For any $\co\in \fN/G$, we have  $d_\co=\displaystyle\max_{\co'\subset \ov{\co}} d_{\sf max}(\co')$.
\end{prop}
\begin{proof}
If $V\subset \ov{\co}$ and $V'$ is a $B$-stable subspace in the closure of $B{\cdot}\{V\}$, then
$\dim V=\dim V'$ and $V'$ is associated with one of the $G$-orbits in $\ov{\co}$.
\end{proof}

Unfortunately, $d_{\sf max}(\co)$ is not easily computable for non-Richardson orbits. Therefore,
this does not provide a quick way to compute $d_\co$ for all nilpotent orbits.
\begin{ex}    \label{ex:mnogo-max}
If  $\p$ is a standard polarisation of $\co\in\fN/G$, then $\p^{nil}\in \cI(\be)_{\co}$
and $\dim\p^{nil}=d_{\sf max}(\co)$. Moreover, all elements of $\cI(\be)_{\co}$ of maximal dimension are 
of this form (Theorem~\ref{thm:main-half}(ii)).
But there can exist other maximal elements of the poset $\cI(\be)_{\co}$ having a smaller dimension. For instance, if $\blb=(2,2,1,1)$ for $\mathfrak{sl}_6$, then $\hat\blb=(4,2)$, $\dim\co(\blb)=16$ and 
$d_{\sf max}(\co(\blb))=8$.
But $\cI(\be)_{\co(\blb)}$ contains also a maximal $\be$-ideal of dimension 5. 
The two nilradicals associated with $\co(\blb)$ are \ 
$  \Yvcentermath1
   \tiny\young(\hfil\hfil\hfil\hfil,\hfil\hfil\hfil\hfil)\ \ \text{\normalsize and } \ 
   \tiny\young(\hfil\hfil,\hfil\hfil,\hfil\hfil,\hfil\hfil)$\ ,
while the 5-dimensional ideal is \ \raisebox{1ex}{$\Yvcentermath1
   \tiny\young(\hfil\hfil\hfil,::\hfil,::\hfil)$} \ (in the usual upper-triangular matrix form).
However, we are not aware of examples of $\co$ such that $\cI(\be)_\co$ contains a minimal element of
a non-minimal dimension.
\end{ex}

If $\co$ is not rigid, e.g. $\co=\Ind_\el^\g(\co')$ for some $\co'\subset\el$, then the knowledge of associated
$\be_\el$-ideals for $\co'$  provides a significant information on $\cI(\be)_\co$. 
\begin{lm}  \label{lm:induction-poset}
Let $\p$ be a standard parabolic subalgebra of $\g$, $\el$ a standard Levi subalgebra of $\p$,  
$\co'\subset \el$ a nilpotent $L$-orbit, and $\co=\Ind_\el^\g(\co')$.
Then $\be_\el=\be\cap\el$ is a Borel subalgebra of $\el$ and
if $\ce'\in \cI(\be_\el)_{\co'}$, then $\ce'\oplus \p^{nil}\in \cI(\be)_\co$. Moreover, if $\ce'$ is a maximal 
element of $\cI(\be_\el)_{\co'}$, then $\ce'\oplus \p^{nil}$ is a maximal element of $\cI(\be)_\co$.
\end{lm}
\begin{proof} The first assertion readily follows from the definition of induction.
We have only to check the last assertion. If $\tilde\ce\in\cI(\be_\el)$ and $\ce'\varsubsetneq \tilde\ce$, then
$\tilde\ce$ is associated with a larger nilpotent $L$-orbit $\tilde\co$, since $\ce'$ was maximal. 
Then Eq.~\eqref{eq:dim-induced} shows that $\dim\Ind_\el^\g(\tilde\co)>\dim\co$. Hence 
$\tilde\ce\oplus \p^{nil}$ is not associated with $\co$.
\end{proof}

\begin{rmk}
It is important in Lemma~\ref{lm:induction-poset} that one can use \un{different} standard parabolics $\p$ with 
conjugate Levi subalgebras. This yields different nilradicals $\p^{nil}$ and hence different (maximal) 
elements of $\cI(\be)_\co$.
\end{rmk}
It follows from Lemma~\ref{lm:induction-poset} that if $\co$ is induced from $(\el,\co')$, then 
\[
    d_{\sf max}(\co)\ge d_{\sf max}(\co') +\dim \p^{nil}=d_{\sf max}(\co') +\frac{1}{2}(\dim\g-\dim\el) .
\]
However, there can be several different ways to induce $\co$ from a smaller orbit (i.e., several essentially 
different pairs $(\el,\co')$) and the numbers $d_{\sf max}(\co') +\dim \p^{nil}$ can differ for different pairs. 
Computations related to this observation suggest the following assertion.

\begin{conj}
If $\co$ is not rigid, then there is always a pair $(\el,\co')$ such that
$\co=\Ind_\el^\g(\co')$ and $d_{\sf max}(\co)= d_{\sf max}(\co') +\dim \p^{nil}$.
\end{conj}
If true, this assertion would reduce the problem of finding $d_{\sf max}(\co)$ to the rigid orbits.
(However, we do not know yet how to determine $d_{\sf max}(\co)$ for the rigid orbits.)
Another simple but useful observation is

\begin{lm}     \label{lm:induced-simple-roots}
Given $\co\in\fN/G$, suppose that  $\ce\in\cI(\be)_\co$ contains simple root spaces $\g^\ap$ 
($\ap\in \eus S$) for some $\eus S\subset\Pi$.
Then $\co$ is induced from a nilpotent orbit in the Levi subalgebra $\el\{\eus S\}$ of the parabolic 
subalgebra $\p\{\eus S\}$.
\end{lm}
\begin{proof}
Since $\ce$ is $\be$-stable and $\g^\ap\subset \ce$ for all $\ap\in \eus S$, we have 
$\g^\gamma\subset\ce$ whenever $\gamma\in\Delta^+$ has a positive $\ap$-coordinate for some 
$\ap\in \eus S$, i.e. $\gamma\curge\ap$. That is, $\p\{\eus S\}^{nil}\subset\ce$. Let $\el\{\eus S\}$ be the standard Levi subalgebra of $\p\{\eus S\}$. As in the proof of Theorem~\ref{thm:main-half}, we have 
$\ce=\p\{\eus S\}^{nil}\oplus \ce'$, where $\ce'$ is a nilpotent $\be\cap\el\{\eus S\}$-stable subspace of 
$\el\{\eus S\}$. If $\co'$ is the dense orbit in $L\{\eus S\}{\cdot}\ce'$, then  $\co=\Ind_\el^\g(\co')$.
\end{proof}

\begin{cl}
The orbit $\co\in\fN/G$ is rigid if and only if\/ $\ce\subset [\ut,\ut]$ for any $\ce\in\cI(\be)_\co$.
\end{cl}

\section{Extreme orbits}
\label{sect:extreme}

\begin{df} A nilpotent orbit $\co=G{\cdot}e$ is said to be {\it extreme}, if $d_{\sf min}(\co) = d_{\sf Dy}(\co)$.
\end{df}
\noindent 
Let $B(G_e)$ be a Borel subgroup of the centraliser $G_e$. If $\ce\in\cI(\be)_\co$ and 
$e\in\ce\cap\co$, then $\dim\ce\ge \dim B{\cdot}e \ge \dim B-\dim B(G_e)$. Hence
$d_{\sf min}(\co)\ge \dim B-\dim B(G_e) $.
It was conjectured in \cite[Conj.\,5.4]{som} that 
\beq    \label{eq:d-min}
d_{\sf min}(\co)=\dim B-\dim B(G_e) . 
\eeq
This was verified for the exceptional simple Lie algebras in~\cite{kaw86} (see also \cite{som}) and 
for the classical Lie algebras in~\cite{fang14}. That is, Eq.~\eqref{eq:d-min} is completely proved 
via a case-by-case verification. It might be very interesting to have a conceptual proof of this 
equality. 

Using properties of $\BZ$-gradings associated with $\tri$-triples, we derive from
Eq.~\eqref{eq:d-min} some practical characterisations of the extreme orbits and classify them.

\begin{prop}  \label{prop:extrim} 
Let\/ $\{e,h,f\}$ be an $\tri$-triple and $\co=G{\cdot}e$. Then 
\\  \indent
{\sf (i)} \ $\dim\g(h;2)+\rk G_e\ge \rk G$ and 
$\co$ is extreme if and only if
\beq  \label{eq:extrim}
            \dim\g(h;2)+\rk G_e=\rk G .
\eeq
\indent {\sf (ii)} \ $\co$ is extreme if and only if the derived group of $G(h;0)$ acts trivially on $\g(h;2)$.
\end{prop}
\begin{proof}
(i) \ Let $G_e=G_e^{red}\times G_e^{u}$ be a Levi decomposition (semi-direct product), i.e., 
$G_e^u$ (resp. $G_e^{red}$) is the unipotent radical (resp. a maximal reductive subgroup) 
of $G_e$. Then $\rk G_e^{red}=\rk G_e$ \ and 
\beq  \label{eq:b_e}
\dim B(G_e)=\dim B(G_e^{red})+\dim G_e^{u} .
\eeq
Without loss of generality, we may assume that $G_e^{red}$ is the stabiliser of $e$ in  
$G(h;0)$~\cite[3.7]{cm}.
Recall that $\dim G_e=\dim\g(h;0)+\dim\g(h;1)$ and also
\beq     \label{eq:G_e-unip}
   \dim G_e^{u}=\dim\g(h;1)+\dim\g(h;2) .
\eeq
Therefore, 
\[ \displaystyle d_{\sf Dy}(\co)=\frac{\dim G{-}\dim \g(h{;}0){-} 2\dim\g(h{;}1)}{2}
=\frac{\dim G{-}\dim G_e{-} \dim\g(h{;}1)}{2} 
\]
and
\begin{multline*}
0 \le d_{\sf Dy}(\co){-} d_{\sf min}(\co)=\frac{\dim G{-}\dim G_e{-} \dim\g(h{;}1)}{2} {-}\dim B{+}\dim B(G_e)\\
=-\frac{\dim T+\dim \g(h;1)}{2} +\dim B(G_e)- \frac{\dim G_e}{2}\\ 
\underset{\eqref{eq:b_e}}{=} \frac{\dim G_e^{u}+\rk G_e{-}\dim T{-}\dim\g(h{;}1)}{2} 
\underset{\eqref{eq:G_e-unip}}{=} \frac{\dim \g(h{;}2)+\rk G_e {-} \dim T}{2} . 
\end{multline*}  

(ii)  Alternatively, we can write $\displaystyle d_{\sf Dy}(\co){=}\frac{\dim G{+}\dim \g(h{;}0)}{2}{-}\dim G_e$ and then similar transformations show that 
\begin{multline*}
0 \le d_{\sf Dy}(\co)- d_{\sf min}(\co)=
\frac{\dim\g(h;0)-\dim T}{2} - (\dim G_e-\dim B(G_e)) \\
= \dim U(G(h;0))-\dim U(G_e^{red}) ,
\end{multline*} 
where $U(\cdot)$ stands for a maximal unipotent subgroup.
Because $G_e^{red}=G(h{;}0)_e \subset G(h{;} 0)$, the equality $d_{\sf Dy}(\co)= d_{\sf min}(\co)$ exactly 
means that the derived group of $G(h;0)$ is contained in $G_e^{red}$.
\end{proof}

\begin{rmk}
The property of being extreme is equivalent to that $\g(h;2)$ has an open $T$-orbit; that is, the roots 
of $\g(h;2)$ are linearly independent.
For the extreme orbits, the Dynkin ideal $\g(h;{\ge}2)$ is both a minimal element and an element of 
minimal dimension in the poset $\cI(\be)_\co$. But $\cI(\be)_\co$ may also contain some other
$\be$-ideals of minimal dimension.
\end{rmk}

Using Eq.~\eqref{eq:b_e} and \eqref{eq:G_e-unip}, one obtains for $e\in\co\cap\g(h;2)$
\begin{multline}   \label{eq:approach}
d_{\sf min}(\co)=\dim B-\dim B(G_e)=\dim B-\dim B(G_e^{red})-\dim G_e^{u}  \\
=\dim B(h;0)+\dim \g(h;{\ge}1)-\dim B(G_e^{red})-\dim\g(h;1)-\dim\g(h;2) \\
=\dim B(h;0)-\dim B(G_e^{red})+ \dim \g(h;{\ge}3) \\ {\color{magenta}\le} \
\dim B(h;0){\cdot}e+\dim \g(h;{\ge}3) .
\end{multline} 
This suggests the following approach to proving Eq.~\eqref{eq:d-min} and also a more precise 
(conjectural)  statement. 
In order to obtain a $\be$-ideal of minimal dimension in $\cI(\be)_\co$, 
we attempt to reduce $\g(h;2)$, while keeping  $ \g(h;{\ge}3)$ intact. If $e\in\co\cap\g(h;2)$ is
such that $B(h;0)_e$ is a Borel subgroup of $G(h;0)_e=G_e^{red}$, then
$B(h;0){\cdot}e$ has the minimal dimension among all $B(h;0)$-orbits in $\co\cap\g(h;2)$ and the 
equality holds in Eq.~\eqref{eq:approach}.
Furthermore, if one can find $e$ such that $\ov{B(h;0){\cdot}e}$ is a subspace, then
$\ov{B(h;0){\cdot}e}\oplus \g(h;{\ge}3)$ is a required $\be$-ideal for $\co$. This provides our conjectural assertion.

\begin{conj}   \label{conj:utochnenie}
Let $\g=\bigoplus_{i\in\BZ}\g(h;i)$ be the grading associated with an $\tri$-triple $\{e,h,f\}$. Let
$B(h;0)$ be a Borel subgroup of $G(h;0)$. Then there is a $B(h;0)$-orbit of minimal dimension in
$G{\cdot}e\cap\g(h;2)$ whose closure is a subspace of\/ $\g(h;2)$.
\end{conj}

Using Jordan normal form, we can prove this conjecture for $\g=\sln$. However, our proof does not readily 
generalise to the other classical series. Therefore, we omit it.

\section{Classification of the extreme orbits}
\label{sect:classification}

\noindent
Our next goal is to classify the extreme orbits in all simple Lie algebras. Clearly, the trivial orbit
$\co=\{0\}$ is extreme, with $d_{\sf Dy}(\co)=d_{\sf min}(\co)=0$. We describe below all non-trivial
extreme orbits.

\begin{ex}   \label{ex:prin+min-nilp}
The principal $\co_{pr}$ and minimal $\co_{min}$ nilpotent orbits are extreme. The corresponding
values are:   

\textbullet \quad $\co_{pr}$:   \ $\dim\g(h;2)=\rk G,\ \rk G_e=0$;

\textbullet  \quad $\co_{min}$:   $\dim\g(h;2)=1,\ \rk G_e=\rk G-1$.
\end{ex}

\begin{thm}   \label{thm:extrem-sl-sp}
{\sf (i)} \ For\/ 
$\GR{A}{n}$, the extreme orbits $\co(\blb)$ correspond to the following partitions:

-- \ if\/ $\g=\sltn$, then $\blb=(2m, 1^{2n-2m})$ with $m=0,1,\dots,n$;

-- \ if\/ $\g=\mathfrak{sl}_{2n+1}$, then $\blb=(2m, 1^{2n+1-2m})$ with $m=0,1,\dots,n$ \ \un{or}\\
\centerline{$\blb=(m,2n+1-m)$ with $m=n+1,\dots,2n$ \  \un{or} \ $\blb=(2n+1)$.}

{\sf (ii)} \ For\/ 
$\GR{C}{n}$, the extreme orbits $\co(\blb)$ correspond to the partitions
$\blb=(2m, 1^{2n-2m})$ with $m=0,1,\dots,n$.
\end{thm}
\begin{proof}
Recall that $\{e,h,f\}$ is an $\tri$-triple and $e\in\co(\blb)$. Let $\eus R(d)$ be the simple 
$\langle e,h,f\rangle$-module of dimension $d+1$. The $h$-weights in $\eus R(d)$ are
$d, d{-}2,\dots,2{-}d,-d$, i.e., the $h$-weight `2' occurs if and only if $d$ is even.
\par
(i) \  It is harmless but notationally more convenient to replace $\sln$ with $\gln$. Then $G=GL_n$.
If $\blb=(\lb_1,\dots,\lb_t)$, then $\rk G_e=t$ and 
\[
\gln =  \bigoplus_{i,j=1}^t \eus R(\lb_i-1)\otimes \eus R(\lb_j-1)
\]
as $\langle e,h,f\rangle$-module. The Clebsch-Gordan formula implies that
the subspace of $h$-weight `$2$' in $\eus R(\lb_i-1)\otimes \eus R(\lb_i-1)$
is of dimension $\lb_i-1$. Hence the sum 
$\bigoplus_{i=1}^t \eus R(\lb_i-1)\otimes \eus R(\lb_i-1)$ yields a subspace of dimension
$\sum_{i=1}^t( \lb_i-1)=\rk G-\rk G_e$ in $\g(h;2)$. Therefore, $\co(\blb)$ is extreme if and only if 
neither of the remaining $\langle e,h,f\rangle$-modules
$\eus R(\lb_i-1)\otimes \eus R(\lb_j-1)$ with $i\ne j$ has $h$-weight `$2$'.
It is easily seen that the $h$-weight `$2$' occurs exactly in the following cases:

a) \ both $\lb_i$ and $\lb_j$ are even;

b) \ both $\lb_i$ and $\lb_j$ are odd and at least one them is bigger than $1$.
\\
Excluding these ``bad'' possibilities, we obtain the required list of $\blb$'s.

(ii) \  A partition $\blb=(\lb_1,\dots,\lb_t)$ of $2n$ represents a nilpotent orbit $\co(\blb)=G{\cdot}e$ in $\spn$ 
if and only if each part of odd size occurs an even number of times~\cite[5.1]{cm}. Then 
\[
   \spn = \left[\bigoplus_{i=1}^t \eus S^2\eus R(\lb_i-1)\right] \oplus
   \left[\bigoplus_{i<j} \eus R(\lb_i-1)\otimes \eus R(\lb_j-1)\right]
\]
as $\langle e,h,f\rangle$-module. Using the Clebsch-Gordan formula, one verifies that
the subspace of $h$-weight `$2$' in $\eus S^2\eus R(\lb_i-1)$
is of dimension $[\lb_i/2]$.

(ii-1) \ {\sl Suppose that  all $\lb_i$ are even}. Then the sum 
$\bigoplus_{i=1}^t \eus S^2\eus R(\lb_i-1)$ already yields a subspace of dimension 
$\frac{1}{2}\sum \lb_i=n=\rk G$ in $\g(h;2)$. Here Eq.~\eqref{eq:extrim} is only satisfied if $t=1$, i.e.,
$\blb=(2n)$ and  $e\in \co_{pr}$.

(ii-2) \ {\sl Suppose that  all $\lb_i$ are odd}. Using a slightly different notation, we can write
\[
   \blb=((2m_1+1)^{2k_1},\dots, (2m_s+1)^{2k_s}) ,
\] 
with $m_1>\dots > m_s\ge 0$. Here 
$G_e^{red}=Sp_{2k_1}\times \ldots \times Sp_{2k_s}$~\cite[Theorem\,6.1.3]{cm} and
$\rk G_e=\sum_{j=1}^s k_j$. Then 
the sum $\bigoplus_{i=1}^t \eus S^2\eus R(\lb_i-1)=\bigoplus_{j=1}^s 2k_j\, \eus S^2\eus R(2m_j)$
yields a subspace of dimension 
$\sum_{j=1}^s 2k_jm_j$ in $\g(h;2)$. Therefore, 
\[
   \dim\g(h;2) +\rk G_e\ge \sum_{j=1}^s k_j(2m_j+1)=n=\rk G .
\] 
Furthermore, if $m_1\ge 1$, then 
there is also a non-trivial contribution to $\g(h;2)$ from a summand of the form 
$\eus R(2m_1)\otimes \eus R(2m_1)$. Hence the only possibility for Eq.~\eqref{eq:extrim} to hold is $s=1$ and $m_1=0$, i.e., 
$\blb=(1^{2n})$ and $e=0$.

(ii-3) \ {\sl The general case}. Separating the even and odd parts of $\blb$, we can write
$\blb=(\blb^{ev}; \blb^{odd})$, which yields a decomposition
\[
   e=e'+e''\in \mathfrak{sp}(V')\oplus \mathfrak{sp}(V'')\subset \mathfrak{sp}(V), 
\]
where $V=V'\oplus V''$ and $\dim V=2n$.
Here $e'\in \co(\blb^{ev})\subset \mathfrak{sp}(V')$ and 
$e''\in \co(\blb^{odd})\subset \mathfrak{sp}(V'')$. Using the equality of 
$\mathfrak{sp}(V')\oplus \mathfrak{sp}(V'')$-modules
\[
  \mathfrak{sp}(V)\simeq \eus S^2(V)=\eus S^2(V')\oplus \eus S^2(V'')\oplus (V'\otimes V'')\simeq
  \mathfrak{sp}(V')\oplus \mathfrak{sp}(V'')\oplus (V'\otimes V'') , 
\]
one readily sees that if $\co(\blb)$ is extreme, then so are $\co(\blb^{ev})$ and $\co(\blb^{odd})$.
Then combining parts (ii-1) and (ii-2) shows that $e'$ is principal in $\mathfrak{sp}(V')$ and $e''=0$. Therefore, if $\dim V'=2m$, then $\blb$ must be of the form $(2m,1^{2n-2m})$, and 
this is also sufficient for $\co(\blb)$ to be extreme.
\end{proof}

\begin{rmk}   \label{rem:of-course-sl-sp}
The partitions in Theorem~\ref{thm:extrem-sl-sp} include those corresponding to 
$\co_{pr}$ and $\co_{min}$. 

\qquad $\sln$ \ \ \ : \ $\co_{pr}=\co(n)$, \quad \ $\co_{min}=\co(2,1^{n-2})$.

\qquad $\spn$ \ : \ $\co_{pr}=\co(2n)$, \quad $\co_{min}=\co(2,1^{2n-2})$.
\end{rmk}

\begin{thm}   \label{thm:extrem-so}  Let $\g$ be an orthogonal Lie algebra. 
\\
{\sf (i)}  For type\/ $\GR{B}{n}$, the nontrivial extreme orbits $\co(\blb)$ correspond to the following partitions $\blb$:
\begin{itemize}
\item $(2,2,1^{2n-3})$   \qquad {\rm [$\co_{min}$]}
\item $(2n-3,2,2)$   \quad {\rm ["intermediate" orbit]}
\item $(2n+1)$   \qquad {\rm [$\co_{pr}$]}
\end{itemize}
{\sf (ii)}  For type\/ $\GR{D}{n}$, the nontrivial extreme orbits $\co(\blb)$ correspond to the following partitions $\blb$:
\begin{itemize}
\item $(2,2,1^{2n-4})$    \qquad {\rm [$\co_{min}$]}
\item $(2n-5,2,2,1)$    \quad  {\rm ["intermediate" orbit]}
\item $(2n-1,1)$     \qquad {\rm [$\co_{pr}$]}
\end{itemize}
\end{thm}
\begin{proof}
Recall that a partition $\blb=(\lb_1,\dots,\lb_t)$ of $n$ represents a nilpotent orbit $\co(\blb)=G{\cdot}e$ in 
$\son$ if and only if each part of even size occurs an even number of times~\cite[5.1]{cm}. Then 
\[
   \son = \left[\bigoplus_{i=1}^t \wedge^2\eus R(\lb_i-1)\right] \oplus
   \left[\bigoplus_{i<j} \eus R(\lb_i-1)\otimes \eus R(\lb_j-1)\right] 
\]
as $\langle e,h,f\rangle$-module. Using the Clebsch-Gordan formula, one verifies that
the subspace of $h$-weight `$2$' in $\wedge^2\eus R(\lb_i-1)$
is of dimension $[\frac{\lb_i-1}{2}]$.
The subsequent argument exploits the same approach as the 
proof of Theorem~\ref{thm:extrem-sl-sp}(ii). Therefore, we omit computational details in the rest of the proof.

{\sl First}, we determine all possible partitions whose all parts are odd. The output is either $(1,\dots,1)$, 
or $(2n+1)$ for $\GR{B}{n}$, or $(2n-1,1)$ for $\GR{D}{n}$.
\\ \indent
{\sl Second}, we determine all possible partitions whose all parts are even. The only output here is
$\blb=(2,2)$.
\\ \indent
{\sl Third}, using the separation $\blb=(\blb^{ev}; \blb^{odd})$ and combining the above two possibilities yields exactly the partitions in the formulation.
\end{proof}

\noindent
Suppose that $\g$ is simple, and let $\theta\in\Delta^+$ be the highest root. Consider the root 
subsystem $\Delta(0)=\{\gamma\in\Delta\mid (\gamma,\theta)=0\}$ and the corresponding Levi
subalgebra $\el=\te\oplus(\bigoplus_{\gamma\in \Delta(0)}\g^\gamma)$. Let $\tilde e$ be a 
principal nilpotent element  of $\el$.

\begin{prop}    \label{prop:extrim-uniform}
If $\theta$ is fundamental, then the $G$-orbit $\tilde\co=G{\cdot}\tilde e$ is extreme, with $\rk G_{\tilde e}=1$.
\end{prop}
\begin{proof}
Since $\theta$ is fundamental, there is a unique $\beta\in\Pi$ such that $(\beta,\theta)\ne 0$, and 
such $\beta$ is long. In particular, $\Pi\setminus\{\beta\}\subset \Delta(0)$ and 
$\rk \Delta(0)=\rk\Delta -1$. Let $\{e_\theta,h,f_\theta\}$ be an $\tri$-triple containing 
$e_\theta\in \g^\theta$ and $f_\theta\in \g^{-\theta}$. For the $\BZ$-grading of $\g$ determined by
$h$, we have:
\[
  \g=\bigoplus_{i=-2}^2 \g(h;i), \ \g(h;0)=\el, \ \g(h;2)=\g^\theta=\langle e_\theta\rangle .
\] 
Here $\el=[\el,\el]\oplus \langle h\rangle$,  
$\g(h;{\ge}0)=\p\{\beta\}$ is the maximal parabolic subalgebra corresponding to $\beta$, and
$\p\{\beta\}^{nil}$ is a Heisenberg Lie algebra.
Let $\{\tilde e,\tilde h,\tilde f\}$ be a principal $\tri$-triple in $[\el,\el]$. Since $[\el,\el]$ is the 
reductive part of the centraliser of $e_\theta$,
the two $\tri$-triples under consideration (pairwise) commute. Without loss of generality, we may also 
assume that $\tilde h\in \te\subset\el$ and $\ap(\tilde h)=2$ for all $\ap\in\Pi\setminus \{\beta\}$. 

Since $\tilde e$ is principal in $\el$, the $\tilde h$-weights in $\el=\g(h;0)$ are even.
If $\el=\bigoplus_{i\in\BZ}\el({\tilde h}; 2i)$, then $\el({\tilde h};0)=\te$ and $\el({\tilde h};2)=
\bigoplus_{\ap\in\Pi\setminus \{\beta\}}\g^\ap$. On the other hand, $\g(h;1)$ is a symplectic
simple $[\el,\el]$-module. Since the triple $\{\tilde e,\tilde h,\tilde f\}\subset [\el,\el]$ is principal, the 
$\tilde h$-weights in $\g(h;1)$ are odd. Recall that $[\tilde h,e_\theta]=0$.
Therefore, $\g({\tilde h};0)=\g^{-\theta}\oplus\te\oplus \g^{\theta}$ and 
$\dim \g({\tilde h};2)=\dim \el({\tilde h};2)= \dim\te-1$. It follows that 
$\g_{\tilde e}^{red}=\langle e_\theta,h,f_\theta\rangle$ and $\rk G_{\tilde e}=1$.

Thus, $\rk G_{\tilde e}+\dim \g({\tilde h};2)=\rk G$, and we are done.
\end{proof}

\noindent
The orbits considered in Proposition~\ref{prop:extrim-uniform} are said to be {\it intermediate\/} and 
denoted by $\co_{imd}$. One easily verifies that the  ``intermediate'' orbits of 
Theorem~\ref{thm:extrem-so} coincide with those occurring in Proposition~\ref{prop:extrim-uniform} 
in the orthogonal case.

\begin{thm}  \label{thm:extrem-excepts} 
If $\g$ is simple and the highest root $\theta$ is fundamental (i.e., $\g\ne\sln$ or $\spn$), then
there are exactly three non-trivial extreme orbits: $\co_{pr}$, $\co_{min}$, and the intermediate orbit
$\co_{imd}$. 
\end{thm}
\begin{proof}   For the series $\GR{B}{n}$ ($n\ge 3$) and $\GR{D}{n}$ ($n\ge 4$), this follows from Theorem~\ref{thm:extrem-so}. 
For the exceptional simple Lie algebras, one consults Elashvili's tables in~\cite{ela76}. Those tables include 
all the required information: the dimension of $\g(h;2)$ and the description of $G_e^{red}$ for all nilpotent 
orbits. 
\end{proof}

\begin{table}[htb]  
\begin{center}
\begin{tabular}{c|ccccccc}   
Algebra &  $\co_{imd}$ & {\sf wDd} of $\co_{imd}$ & $\dim\co_{imd}$ & $d(h;1)$ & $d(h;2)$ & 
$\rk G_e$ & $d_{\sf Dy}$  \\ \hline
$\GR{G}{2}$ & $\GRt{A}{1}$\rule{0ex}{2.5ex} & 1$\Lleftarrow$0 & 8 & 2 & 1 & 1 & 3  \\
$\GR{F}{4}$ & $\GR{C}{3}$ & 2-1$\Leftarrow$0-1 & 42 & 4 & 3 & 1 & 19 \\
$\GR{E}{6}$ & $\GR{A}{5}$ & \begin{E6}{2}{1}{0}{1}{2}{1}\end{E6} 
 & 64 & 6 & 5 & 1 & 29 \\
$\GR{E}{7}$ & $\GR{D}{6}$ & \begin{E7}{2}{2}{1}{0}{1}{2}{1}\end{E7}
& 118 & 6 & 6 & 1 & 56  \\
$\GR{E}{8}$ & $\GR{E}{7}$ & \begin{E8}{2}{2}{2}{1}{0}{1}{2}{1}\end{E8}
& 232 & 6 & 7 & 1  & 113 \\  \hline
\end{tabular} \vskip1ex
\caption{The intermediate orbits in the exceptional Lie algebras}
\label{table:except}
\end{center}
\end{table}
\noindent
In Table~\ref{table:except}, $d(h;i)=\dim\g(h{;}i)$, $i=1,2$, and $d_{\sf Dy}=d_{\sf Dy}(\co_{imd})$.
The second column contains the usual notation for nilpotent orbits in the 
exceptional Lie algebras (see~\cite[8.4]{cm}) and the third column shows the 
{\it weighted Dynkin diagram} (={\sf wDd}) of $\co_{imd}$.

\section{Lonely orbits}
\label{sect:lonely}

\begin{df} A nilpotent orbit $\co$ is said to be {\it lonely}, if $\#(\cI(\be)_{\co})=1$. In other words, 
the Dynkin ideal  is the only element of $\cI(\be)_{\co}$. 
\end{df}
\noindent 
Clearly, a lonely orbit is extreme. Therefore, to classify the lonely orbits, one has to explore the explicit list of extreme orbits from Section~\ref{sect:classification}.

\begin{ex}
The principal nilpotent orbit $\co_{pr}$ is always lonely, since $\ut$ is the only element of 
$\cI(\be)_{\co_{pr}}$.
\end{ex}

\begin{lm}     \label{lm:O-min-neodinoka}
If there is a {\bfseries long} root $\ap\in\Pi$ such that $(\ap,\theta)\ne 0$ and $\ap\ne\theta$, then 
$\co_{min}$ is not lonely.
\end{lm}
\begin{proof}
Clearly, the root space $\g^\theta $ is the only minimal element of $\cI(\be)_{\co_{min}}$.
Moreover, if $\ap$ is long and $\ap\ne\theta$, then the $2$-dimensional $\be$-ideal 
$\langle \g^\theta, \g^{\theta-\ap}\rangle$ also belongs to  $\cI(\be)_{\co_{min}}$.
\end{proof}
\noindent
This lemma applies to the minimal orbits in all simple Lie algebras except $\spn$ ($n\ge 1$). Note that
$\mathfrak{sp}_2=\tri$. 
\begin{lm}    \label{lm:non-even-Rich}
If $\co\subset\fN$ is Richardson, but not even, then it is not lonely.
\end{lm}
\begin{proof}
Since $\co\subset\fN$ is Richardson, $d_{\sf max}(\co)=\frac{1}{2}\dim\co$. If $\co$ is not even, then
$\g(h;1)\ne 0$  and hence $d_{\sf max}(\co)-d_{\sf Dy}(\co)=\frac{1}{2}\dim\g(h;1)>0$.
\end{proof}

Lemma~\ref{lm:non-even-Rich} applies to all non-even orbits in $\sln$,  but a stronger assertion is true!
\begin{prop}   \label{pr:odinokie-sl}
In $\g=\sln$, the only nonzero lonely orbit is $\co_{pr}$. 
\end{prop}
\begin{proof}
Here all nonzero orbits $\co(\blb)$ are Richardson. If $\p$ is a standard polarisation of $\co(\blb)$, then 
$\p^{nil}\in \cI(\be)_{\co(\blb)}$. Therefore, if $\co(\blb)$ is lonely, then it has a {\bf unique} standard 
polarisation. This means that all the diagonal blocks of $\p$ must be equal, i.e., $\blb$ is rectangular.
However, $\co(\blb)$ must also be extreme. Comparing with Theorem~\ref{thm:extrem-sl-sp}(i), we get only 
two possibilities:  either $\blb$ is a single row (i.e., $\co(\blb)=\{0\}$) or  $\blb$ is a single column (i.e., 
$\co(\blb)=\co_{pr}$).
\end{proof}

\begin{ex} For types $\GR{B}{n}$ ($n\ge 3$) and $\GR{F}{4}$, the intermediate orbits 
are Richardson  and not even. Hence they are not lonely. In both cases, the semisimple part of $\el$ is of type
$\GR{A}{2}$. The {\sf wDd} of $\co_{imd}(\GR{F}{4})$ is given in Table~\ref{table:except}, and the {\sf wDd} of
$\co_{imd}(\GR{B}{n})$ is $(\underbrace{2\dots 2}_{n-3}10{\Rightarrow}1)$. \\
Thus, for $\GR{A}{n}$ ($n\ge 1$), $\GR{B}{n}$ ($n\ge 3$), 
and $\GR{F}{4}$, the only lonely orbit is $\co_{pr}$.
\end{ex}

\subsection{Lonely orbits in the symplectic Lie algebra.}
We use the embedding $\spn\subset\sltn$ such that
\beq    \label{eq:sympl-embed}
\spn=\left\{\begin{pmatrix} \eus A & \eus B \\ \eus C & -\eus{\hat A} \end{pmatrix} \mid \eus A,\eus B,\eus C\in\gln,
\quad \eus B=\hat{\eus B}, \eus C=\hat{\eus C}\right\} ,
\eeq 
where $\eus A\mapsto\hat{\eus A}$ is the transpose with respect to the antidiagonal. Then our fixed Borel
subalgebra $\be(\spn)$ is the set of symplectic upper-triangular matrices. 
For the roots of $\spn$, we use the standard notation in which 
$\Pi=\{\ap_1=\esi_1-\esi_2,\dots,\ap_{n-1}=\esi_{n-1}-\esi_n, \ap_n=2\esi_n\}$.

Let  $\co_m(n)$ denote the extreme orbit in $\spn$ that corresponds to $\blb=(2m,1^{2n-2m})$, 
$1\le m\le n$.  In particular, $\co_1(n)=\co_{min}$ and $\co_n(n)=\co_{pr}$. 
Our goal is to prove that all these orbits are lonely. Before starting the proof, we gather some
relevant data.

{\sf\bfseries 1}$^o$. For $m< n$, the {\sf wDd} of $\co_m(n)$ is 
$(\underbrace{2\dots 2}_{m-1}1\underbrace{0\dots 0{\Leftarrow}0}_{n-m})$.
For $m=n$, we have $\co_n(n)=\co_{pr}$ and hence {\sf wDd} is $(2\dots 2{\Leftarrow}2)$. The orbit
$\co_m(n)$ consists of (symplectic nilpotent) $2n\times 2n$ matrices of rank  $2m-1$.

{\sf\bfseries 2}$^o$. Using the above weighted Dynkin diagram, one readily determines the $\BZ$-grading associated with  
$\{e,h,f\}$, where $e\in\co_m(n)$. For instance, $\dim\g(h;1)=2n-2m$, $\dim\g(h;2)=m$, and 
$d_{\sf Dy}(\co_m(n))=\dim\g(h;{\ge}2)=(m-1)(2n-m+1)+1$.

{\sf\bfseries 3}$^o$. The orbit $\co_m(n)$ is rigid if and only if $m=1$, and for $m\ge 2$, $\co_m(n)$ is 
induced from
$(\GR{C}{n-1}, \co_{m-1}(n{-}1))$. Note that $\spn$ contains a unique standard Levi subalgebra of (semisimple) type
$\GR{C}{k}$ with $k<n$, hence the above induction is well defined. Expanding the chain of induction, 
we see that $\co_m(n)$ is induced from the rigid orbit $\co_{min}=\co_{1}(n-m+1)$ in the unique Levi 
subalgebra $\el\subset\spn$ such that  $[\el,\el]=\mathfrak{sp}_{2(n-m+1)}$. Actually, the pair
$(\GR{C}{n-m+1}, \co_{1}(n-m+1))$ represents the {\bf only} possibility to obtain $\co_m(n)$ via induction 
from a rigid orbit in a Levi subalgebra of $\spn$.


\begin{lm}
The extreme orbit $\co_1(n)$ is lonely. 
\end{lm}
\begin{proof}
For $\spn$,  $\ap_1=\esi_1-\esi_2$ is the only simple root that is not orthogonal to $\theta=2\esi_1$, 
and $\ap_1$ is {\bf short}. Since $\co_1(n)$ consists of matrices of rank $1$ and
the unique 2-dimensional $\be$-ideal $\langle \g^\theta, \g^{\theta-\ap_1}\rangle$ contains already
matrices of rank $2$, the root space $\g^\theta$ is the only $\be$-ideal associated with $\co_1(n)$.
\end{proof}

Since $\co_n(n)$ is always lonely, this implies that all extreme orbits for $n=1,2$ are lonely.

\begin{thm}    \label{thm:lonely-sp}
The extreme orbits $\co_m(n)$ are lonely for all $m\le n$.
\end{thm}
\begin{proof}
Arguing by induction on $n$, we may assume that all orbits $\co_m(n')$ with $n'<n$ and $m\le n'$ 
are lonely. The cases with $n=1,2$ form the base of induction. 
 
Since $\co_n(n)$ and $\co_1(n)$ are lonely, we may assume that $1<m<n$. 
We know that $d_{\sf min}(\co_m(n))=d_{\sf Dy}(\co_m(n))=(m-1)(2n-m+1)+1$. 
The minimal roots of the Dynkin ideal $\ce_{\sf Dy}(\co_m(n))$ are $\ap_1,\dots,\ap_{m-1},2\esi_m$.
The corresponding root spaces are marked by `$\ast$' in Fig.~\ref{pikcha_sp}.
Here $2\esi_m$ is the highest root of the regular simple subalgebra of type $\GR{C}{n-m+1}$ whose 
simple roots are $\ap_m,\dots,\ap_n$.  This structure of $\ce_{\sf Dy}(\co_m(n))$ and 
Fig.~\ref{pikcha_sp} visibly demonstrate
that $\co_m(n)$ is induced from the minimal nilpotent orbit in $\mathfrak{sp}_{2(n-m+1)}$ via the use 
of the standard parabolic subalgebra $\p\{\ap_1,\dots,\ap_{m-1}\}$.

\begin{figure}[htb]
\setlength{\unitlength}{0.025in}
\begin{center}
\begin{picture}(111,115)(-3,0)

\multiput(0,0)(110,0){2}{\line(0,1){110}}
\multiput(0,0)(0,110){2}{\line(1,0){110}}
\qbezier[50](0,55),(55,55),(110,55)
\qbezier[50](55,0),(55,55),(55,110)    
\qbezier[25](54,54),(82,82),(110,110)    
\qbezier[10](85,90),(85,100),(85,110)
\qbezier[10](60,90),(60,100),(60,110)
\qbezier[12](84,90),(97,90),(110,90)

\qbezier[35](20,20),(55,20),(90,20)   
\qbezier[35](20,20),(20,55),(20,90)    
\put(85,90){\line(1,0){5}}  \put(90,85){\line(0,1){5}}
\thicklines
{\color{magenta}
\multiput(5,105)(5,-5){4}{\line(0,1){5}}
\multiput(5,105)(5,-5){3}{\line(1,0){5}}
\multiput(90,20)(5,-5){4}{\line(1,0){5}}
\multiput(105,5)(-5,5){3}{\line(0,1){5}}
\put(20,90){\line(1,0){65}}  \put(85,85){\line(1,0){5}}
\put(90,20){\line(0,1){65}}  \put(85,85){\line(0,1){5}}
}
\thinlines
\multiput(6,106)(5,-5){4}{$\ast$}
\multiput(61,106)(5,-5){5}{$\star$}
\put(86,86){$\ast$}

\put(85,110){$\overbrace
{\mbox{\hspace{25\unitlength}}}^{m}$}   
\put(60,110){$\overbrace
{\mbox{\hspace{25\unitlength}}}^{m}$}   

\put(115,99){{\footnotesize $m-1$}}   
\put(116,26){{\footnotesize $n$}}   

\put(-5,95){\vector(1,1){11}}     \put(-10,93){\footnotesize $\ap_1$}     
\put(10,80){\vector(1,1){11}}     \put(5,78){\footnotesize $\ap_{m-1}$}       
\put(70,75){\vector(1,1){11}}     \put(61,71){\footnotesize $\esi_m{+}\esi_{m{+}1}$}       
\put(100,77){\vector(-1,1){10}}     \put(100,73){\footnotesize $2\esi_m$}       
\put(-20,53){$\spn=$}
\end{picture}
\lefteqn{\raisebox{98\unitlength}%
{$\left. {\parbox{1pt}{\vspace{20\unitlength}}}
\right\}$}}%
\lefteqn{\raisebox{25\unitlength}%
{$\left. {\parbox{1pt}{\vspace{55\unitlength}}}
\right\}$}}%

\caption{The Dynkin ideal for the extreme orbit $\co_m(n)$ in $\spn$}   \label{pikcha_sp}
\end{center}
\end{figure}

Assume that there is another $\ce\in \cI(\be)_{\co_m(n)}$.  Since $\co_m(n)$ is extreme, we must 
have $\dim\ce\ge (m-1)(2n-m+1)+1$.

\textbullet\quad If $\Delta(\ce)$ contains a simple root $\ap_k$, then $\co_m(n)$ is induced from a nilpotent orbit in 
$\el\{\ap_k\}=\mathfrak{gl}_k\oplus \mathfrak{sp}_{2n-2k}$, see Lemma~\ref{lm:induced-simple-roots}.
Since $\mathfrak{gl}_k$ has no nontrivial rigid orbits and the only rigid orbit from which $\co_m(n)$
can be induced is contained in $\mathfrak{sp}_{2(n-m+1)}$, we must have $k\le m-1$ and
$\ap_1,\dots,\ap_{k-1}$ also belong to $\Delta(\ce)$. Then $\ce=\p\{\ap_1,\dots,\ap_k\}^{nil}\oplus \ce'$, where 
$\ce'$ is an ideal associated with the orbit $\co_{m-k}(n-k)$ in 
$[\el\{\ap_1,\dots,\ap_k\},\el\{\ap_1,\dots,\ap_k\}]\simeq \mathfrak{sp}_{2(n-k)}$. By the induction 
assumption, $\co_{m-k}(n-k)$ is lonely. Hence $\ce'$ is necessarily the Dynkin ideal for $\co_{m-k}(n-k)$ and then
$\ce$ is the Dynkin ideal for $\co_m(n)$.

\textbullet\quad  
It remains to handle the case in which $\Delta(\ce)$ contains no simple roots, i.e., 
$\ce\subset [\ut,\ut]=:\ut'$. The dense $Sp_{2n}$-orbit meeting $\ut'$ corresponds to the partition 
$\boldsymbol{\beta}=(n,n)$. For $2m > n$, the largest part of $\boldsymbol{\beta}$ is fewer than that 
of $\blb$. Then the description of the closure relation in $\fN/Sp_{2n}$ via partitions~\cite[6.2]{cm} 
shows that $\co_m(n)\not\subset\ov{\co(\boldsymbol{\beta})}$, i.e., it
cannot meet $\ut'$.   Hence $\ce$ cannot lie in $\ut'$ and thereby  
$\co_m(n)$ is lonely whenever $m> n/2$.

\textbullet\quad Finally, we assume that $m\le n/2$ and show that if $\ce$ is a
$\be$-ideal such that $\ce\subset \ut'$ and $\rk A\le 2m-1$ for any $A\in \ce$, then 
$\dim\ce < (m-1)(2n-m+1)+1=d_{\sf Dy}(\co_m(n))$. The subsequent argument exploits (1) the chosen matrix form of $\spn$ and (2) the fact that $\Delta(\ce)$ is an upper ideal of $\Delta^+$.

{\bf --} \ If $\esi_m- \esi_j$ ($j\ge m+1$) belongs to $\Delta(\ce)$, then $\ce$ contains a matrix of rank $\ge 2m$.
For the same reason, $\esi_m+\esi_j$ ($j\ge 2m$) does not belong to $\Delta(\ce)$, too.
Therefore, all roots of $\ce$ are of the form $\esi_i-\esi_j$  ($i\le m-1$ and $j>i$) or 
$\esi_i+\esi_j$ ($i\le m-1$) or $\esi_s+\esi_t$ ($s\le t\le 2m-1$). 

{\bf --} \ If the  positive roots $\esi_1+\esi_{2m}, \esi_2+\esi_{2m-1},\dots,\esi_m+\esi_{m+1}$ belong
to $\Delta(\ce)$, then $\ce$ contains a matrix of rank $2m$. (These roots are marked by `$\star$' in Fig.~\ref{pikcha_sp}, and since the submatrix $\eus B$ in Eq.~\eqref{eq:sympl-embed} is symmetric with respect to the
antidiagonal, this gives rise to $2m$ nonzero entries!)
Hence at least one of them does not belong to $\Delta(\ce)$. 
Take the minimal $i$ such that $\esi_i+\esi_{2m+1-i}\not\in\Delta(\ce)$. Since $\ce$ is $B$-stable (i.e., $\Delta(\ce)$ is an upper ideal of $(\Delta^+,\curle)$), 
$\Delta(\ce)$ is contained in 
\[
    \bigl(\{\esi_l-\esi_j\mid l<j \ \& \ l\le i-1\}\cup \{\esi_s+\esi_t\mid s\le t\le 2m-i \}\bigr)\setminus 
    \{\ap_1,\dots,\ap_{m-1}\}
\]
and a direct calculation shows that $\dim\ce \le (i-1)(2n-2m+\frac{i-2}{2})+ \frac{(2m+1-i)(2m-i)}{2}$.
Now the desired assertion that $\dim\ce <d_{\sf Dy}(\co_m(n))$ is implied by the next lemma.

Thus, our assumption that there is another $\ce\in\cI(\be)_{\co_m(n)}$ leads to a contradiction, and we are done.
\end{proof}

\begin{lm}  \label{lm:ner-vo}
If\/ $1\le i\le m\le n/2$ and $m\ge 2$, then  
\[
(i-1)(2n-2m+\frac{i-2}{2})+ \frac{(2m+1-i)(2m-i)}{2}< (m-1)(2n-m+1)+1 .
\]
\end{lm}
\begin{proof}
The difference of two quantities in question can be written as \\
$(m-i)[2(n-2m)+m+i-1]+i-1$, which is positive.
\end{proof}


\subsection{Lonely orbits in the even orthogonal and exceptional Lie algebras.}
Write $\co_{imd}(\g)$ for the intermediate orbit in $\g$ with the fundamental $\theta$.
The algebras of type $\GR{D}{n}$ ($n\ge 5$) and $\GR{E}{n}$ ($n=6,7,8$) have 
a unique standard Levi subalgebra $\el$ with $[\el,\el]=\mathfrak{so}_8$ and their intermediate orbits are 
induced from $(\el, \co_{imd}(\mathfrak{so}_8))$. Moreover, this is the {\bf only} possibility to induce them 
from a rigid orbit in a proper Levi subalgebra. 
Using this property, we prove below that the intermediate orbits in all
$\GR{D}{n}$ and $\GR{E}{n}$ are lonely. It is also easily seen directly that $\co_{imd}(\GR{G}{2})$ is lonely.

We use the embedding $\sone\subset\sltn$ such that $\sone$ consists of the skew-symmetric matrices 
with respect to the antidiagonal.
For the simple roots of $\sone$, we use the standard notation in which 
$\ap_i=\esi_i-\esi_{i+1}$ ($i\le n-1$) and $\ap_n=\esi_{n-1}+\esi_n$. 

\begin{lm}
The intermediate orbit in\/ $\mathfrak{so}_8$ is lonely.
\end{lm}
\begin{proof}
If $\co=\co_{imd}(\mathfrak{so}_8)$, then $\dim\co=16$ and 
$d_{\sf Dy}(\co)=5$.  The minimal roots of $\Delta(\ce_{\sf Dy}(\co))$ are
$\ap_1+\ap_2+\ap_3$,
$\ap_1+\ap_2+\ap_4$, $\ap_2+\ap_3+\ap_4$. This $\be$-ideal is abelian, i.e., 
$[\ce_{\sf Dy}(\co),\ce_{\sf Dy}(\co)]=0$. The orbit $\co$ is spherical~\cite[(4.4)]{manuscr}.  Actually, it is {\bf the} 
maximal spherical nilpotent orbit.
By \cite[Prop.\,4.1]{pr2}, if $\g$ is simply-laced and $\co$ is spherical, then any element of $\cI(\be)_\co$ 
is an abelian $\be$-ideal. Here $\ce_{\sf Dy}(\co)$ is a maximal abelian ideal. There are also three 
other maximal abelian ideals of dimension $6$ (their minimal roots are $\ap_1,\ap_3,$ and $\ap_4$, 
respectively), and three other non-maximal abelian ideals of dimension $5$ (their minimal roots are 
$\ap_1+\ap_2,\ap_2+\ap_3,$ and $\ap_2+\ap_4$, respectively). But all these abelian ideals are associated with smaller nilpotent orbits
(of dimension 12). Hence $\ce_{\sf Dy}(\co)$ is the only element of $\cI(\be)_\co$ and $\co$ is lonely.
\end{proof}

Set $\co_n=\co_{imd}(\sone)$, $n\ge 4$. We know that $\co_4$ is lonely and
$\co_n$ corresponds to the partition $\blb_n=(2n-5,2,2,1)$.  The {\sf wDd} of $\co_n$ is 
\ \ \raisebox{-1ex}{\begin{Dn}{2}{2}{1}{0}{1}{1}\end{Dn}} and the minimal roots of 
$\Delta(\ce_{\sf Dy}(\co_n))$ are
$\ap_1,\dots,\ap_{n-4}, \ap_{n-3}+\ap_{n-2}+\ap_{n-1},$
$\ap_{n-3}+\ap_{n-2}+\ap_n,$ and $\ap_{n-2}+\ap_{n-1}+\ap_n$.

\begin{figure}[htb]
\setlength{\unitlength}{0.025in}
\begin{center}
\begin{picture}(102,105)(-3,0)

\multiput(0,0)(100,0){2}{\line(0,1){100}}
\multiput(0,0)(0,100){2}{\line(1,0){100}}
\qbezier[50](0,50),(50,50),(100,50)
\qbezier[50](50,0),(50,50),(50,100)
\qbezier[25](50,50),(75,75),(100,100)     
\qbezier[35](30,30),(50,30),(70,30)   \qbezier[20](70,45),(70,53),(70,70)
\qbezier[35](30,30),(30,50),(30,70)   \qbezier[20](45,70),(53,70),(70,70)
\qbezier[25](30,70),(30,85),(30,100)

\thicklines
{\color{magenta}
\multiput(5,95)(5,-5){6}{\line(0,1){5}}
\multiput(5,95)(5,-5){5}{\line(1,0){5}}
\multiput(95,5)(-5,5){6}{\line(1,0){5}}
\multiput(95,5)(-5,5){5}{\line(0,1){5}}
\put(30,70){\line(1,0){15}}     \put(45,65){\line(0,1){5}}
\put(45,65){\line(1,0){10}}  
\multiput(55,60)(5,-5){2}{\line(0,1){5}}
\put(70,30){\line(0,1){15}}     \put(65,45){\line(1,0){5}}
\put(65,45){\line(0,1){10}}  
\multiput(55,60)(5,-5){2}{\line(1,0){5}}
}
\thinlines
\multiput(6,96)(5,-5){6}{$\ast$}
\multiput(46,66)(5,0){2}{$\ast$}
\put(56,61){$\ast$}

\put(0,100){$\overbrace
{\mbox{\hspace{30\unitlength}}}^{n-4}$}   
\put(30,100){$\overbrace
{\mbox{\hspace{20\unitlength}}}^{4}$}   

\put(108,23){{\footnotesize $n$}}   

\put(-5,85){\vector(1,1){11}}     \put(-10,83){\footnotesize $\ap_1$}     
\put(20,60){\vector(1,1){11}}     \put(15,58){\footnotesize $\ap_{n-4}$}       
\put(-20,48){$\sone=$}
\end{picture}
\lefteqn{\raisebox{23\unitlength}%
{$\left. {\parbox{1pt}{\vspace{50\unitlength}}}
\right\}$}}%

\caption{The Dynkin ideal for the extreme orbit $\co_n=\co_{imd}(\sone)$}   \label{pikcha_so}
\end{center}
\end{figure}

\noindent
This structure of $\ce_{\sf Dy}(\co_n)$ in Fig.~\ref{pikcha_so} visibly demonstrates that $\co_n$ is induced 
from $(\mathfrak{so}_8, \co_4)$ via the use of the standard parabolic subalgebra 
$\p\{\ap_1,\dots,\ap_{n-4}\}$. The central square represents $\mathfrak{so}_8$, and the
three marked roots inside it  are the minimal roots of $\ce_{\sf Dy}(\co_4)$. Using the {\sf wDd} or 
Fig.~\ref{pikcha_so}, one computes that $d_{\sf Dy}(\co_n)=n^2-n-7$ ($n\ge 4$).

\begin{prop}    \label{prop:lonely-Dn}
If $\co_5$ is lonely, then so are all $\co_n$ with $n\ge 6$. 
\end{prop}
\begin{proof}
Arguing by induction on $n$, we may assume that $\co_4,\dots,\co_{n-1}$ are lonely and $n\ge 6$.
Suppose that $\ce\in\cI(\be)_{\co_n}$.

\textbullet\quad If $\Delta(\ce)$ contains an $\ap_k\in\Pi$, then $\co_n$ is induced from a nilpotent 
$L\{\ap_k\}$-orbit $\co'$ in 
$\el\{\ap_k\}=\mathfrak{gl}_k\oplus \mathfrak{so}_{2n-2k}$, see Lemma~\ref{lm:induced-simple-roots}.
Since $\mathfrak{gl}_k$ has no nontrivial rigid orbits and the only rigid orbit from which $\co_n$ can be 
induced is contained in the unique standard Levi of semisimple type $\GR{D}{4}$, we must have 
$k\le n-4$ and $\ap_1,\dots,\ap_{k-1}$ also belong to $\Delta(\ce)$. Then 
$\ce=\p\{\ap_1,\dots,\ap_k\}^{nil}\oplus \ce'$, where $\ce'$ is an ideal associated with $\co_{n-k}$ 
in $[\el\{\ap_1,\dots,\ap_k\},\el\{\ap_1,\dots,\ap_k\}]\simeq \mathfrak{so}_{2(n-k)}$ and 
$n{-}k\ge 4$. By the induction assumption, $\co_{n-k}$ is lonely. Hence $\ce'$ is the Dynkin ideal for 
$\co_{n-k}$ and then $\ce$ is the Dynkin ideal for $\co_n$.

\textbullet\quad Assume that $\ce\subset \ut':=[\ut,\ut]$. The dense orbit in $SO_{2n}{\cdot}\ut'$ 
corresponds to the partition 
\begin{center}
$\boldsymbol{\beta}=\begin{cases}   (2m-1,2m-1,1,1),  & \text{ if } \ n=2m \\
    (2m+1,2m-1,1,1), & \text{ if } \ n=2m+1 \end{cases}$ \ .
\end{center} 
For $n\ge 6$, the largest part of $\boldsymbol{\beta}$ is fewer than that of $\blb_n$.
Then the description of the closure relation in $\fN/SO_{2n}$ via partitions~\cite[6.2]{cm} shows
that $\co_n\not\subset \ov{\co(\boldsymbol{\beta})}$, i.e.,
$\co_n$ cannot meet $\ut'$.   Hence $\ce$ cannot lie in $\ut'$ and thereby  
$\co_n$ is lonely whenever $n\ge 6$.
\end{proof}

To complete the $\sone$-case, we have only to prove the following assertion.

\begin{lm}   \label{lm:D5}
The intermediate orbit $\co_5$ in $\mathfrak{so}_{10}$ is lonely.
\end{lm}
\begin{proof}  Let $\ce\in \cI(\be)_{\co_5}$.
Repeating the argument of Proposition~\ref{prop:lonely-Dn}, we conclude that either 
$\ce=\ce_{\sf Dy}(\co_5)$ or
$\ce\subset \ut'$. For $n=5$, the partition for the dense $SO_{10}$-orbit in $\ut'$ is
$(5,3,1,1)$ and the partition for $\co_5$ is $(5,2,2,1)$. Hence $\co_5$ does meet $\ut'$. 
Here $\dim\ut'=15$ and $d_{\sf Dy}(\co_5)=13$. Therefore $\dim\ce\ge 13$ and one verifies 
directly that all $B$-stable subspaces of $\ut'$ of dimension $13,14$ are associated with some other 
orbits.
\end{proof}

\begin{prop}     \label{prop:except-lonely}
The intermediate orbits in $\GR{E}{n}$ ($n=6,7,8$) are lonely. 
\end{prop}
\begin{proof}
The sequence $\GR{E}{5}=\GR{D}{5},\GR{E}{6},\GR{E}{7},\GR{E}{8}$ can be regarded as an ``exceptional 
series'', via the natural inclusions of the Dynkin diagrams. We argue by induction on $n$, using 
Lemma~\ref{lm:D5} as the base. As in Theorem~\ref{thm:lonely-sp} or Proposition~\ref{prop:lonely-Dn},
in all three cases, the uniqueness of induction from a rigid orbit  allows us to reduce the problem to the assertion that an ideal $\ce\in\cI(\be)_{\co_{imd}(\g)}$ cannot belong to $\ut'$.
The dense $G$-orbits $\co'$ in $G{\cdot}\ut'$ are:
\\  \indent
for $\GR{E}{6}: \  \co'=\GR{E}{6}(a_3)$; \quad 
for $\GR{E}{7}: \  \co'=\GR{E}{6}(a_1)$; \quad 
for $\GR{E}{8}: \  \co'=\GR{E}{8}(a_4)$. 

\noindent 
The explicit description of the closure relation~\cite[Tables]{spalt82} shows that,
for $n=7$ or $8$,  $\co_{imd}(\g)$ does {\bf not} belong to the closure of $\co'$,  which 
immediately discards the possibility $\ce\subset\ut'$. The situation for $\GR{E}{6}$ is similar with that 
for $\GR{D}{5}$ (Lemma~\ref{lm:D5}).
The orbit $\co_{imd}(\GR{E}{6})$ does belong to the closure of $\co'$. 
But here $\dim\ut'=30$ and $d_{\sf Dy}(\co_{imd}(\GR{E}{6}))=29$, hence $\dim\ce\ge 29$. Then an 
explicit verification shows that all $B$-stable subspaces of $\ut'$ of codimension $1$ are associated 
with some other orbits.
\end{proof}

Gathering together the previous results, we obtain the following classification.

\begin{thm}   \label{thm:all-lonely}  
For any simple Lie algebra $\g$, the orbit $\co_{pr}$ is lonely. Furthermore, 
\begin{itemize}
\item[$\diamond$] \  For\/ $\GR{A}{n}$ ($n\ge 1$), $\GR{B}{n}$ ($n\ge 3$), 
and $\GR{F}{4}$, the only lonely orbit is $\co_{pr}$;
\item[$\diamond$] \ For\/ $\GR{D}{n}$ ($n\ge 4$), $\GR{E}{n}$ ($n=6,7,8$), and $\GR{G}{2}$, 
the lonely orbits are $\co_{pr}$ and $\co_{imd}$;
\item[$\diamond$] \ All extreme orbits $\co_m(n)$ in\/ $\GR{C}{n}$ ($n\ge 2$) are lonely.
\end{itemize}
\end{thm}

\noindent 
As a by-product of our classifications in Sections~\ref{sect:classification} and~\ref{sect:lonely}, we note the following property.

\begin{prop}    \label{prop:d-min=d-max}
If \ $\co\in \fN/G$ and $d_{\sf min}(\co)=d_{\sf max}(\co)$, then $\co$ is lonely.
\end{prop}
\begin{proof}
Indeed, such an $\co$ is extreme. And if $\co$ is not lonely, then either it is Richardson  and not even, where
Lemma~\ref{lm:non-even-Rich} applies,
or $\co=\co_{min}$ with $\g\ne\spn$, where Lemma~\ref{lm:O-min-neodinoka} applies.
\end{proof}

This means that the case in which $\cI(\be)_\co$ consists of {\sl several\/} ideals of 
one and the same dimension is impossible. It might be interesting to find an {\it a priori\/} explanation.

\section{Some anomalies and problems}
\label{sect:anomalies}

\noindent
In this section, we compare the numbers $d_\co$ and $d_{\sf max}(\co)$ for different nilpotent orbits.
Clearly if $\co_1 \subset \ov{\co_2}$, then $d_{\co_1}\le d_{\co_2}$. Furthermore, if $\co_1\ne\co_2$ and 
$\co_2$ is Richardson, then $d_{\co_1}\le \frac{1}{2}\dim\co_1<\frac{1}{2}\dim\co_2=d_{\co_2}$. 
This is a sort of natural behaviour that one could expect {\it a priori}. However, passing to the numbers 
related to $B$-stable nilpotent subspaces, we encounter strange anomalies.  For, it can happen that, for one of the two orbits, $d_{\sf max}(\co)$ is considerably less that $d_\co$.

\begin{ex}   \label{ex:anomaly}
We provide examples of two orbits $\co_1 \subset \ov{\co_2}$ such that
$d_{\sf max}(\co_1)> d_{\sf max}(\co_2)$. Using this, we also show that $d_{\sf max}(\co_2)< d_{\co_2}$.
We first describe a general idea that allows us to detect such ``bad" pairs of nilpotent orbits. 
Suppose that $\dim\co_2-\dim\co_1=2$ and $\co_1$ is Richardson, whereas $\co_2$ is not. 
Then 
\[
    d_{\sf max}(\co_2) \le \frac{1}{2}\dim\co_2 -1=\frac{1}{2}\dim\co_1=d_{\sf max}(\co_1) .
\]
In some cases, one can show that $d_{\sf max}(\co_2)$ is strictly less than $\frac{1}{2}\dim\co_2 -1$, which lead to the ``desired'' inequality $d_{\sf max}(\co_1)> d_{\sf max}(\co_2)$.
Moreover, since $d_{\sf max}(\co_1)\le d_{\co_2}$, this also
implies that  $d_{\sf max}(\co_2)< d_{\co_2}$, see (7.1.2) below. 
Even if $\dim\co_2-\dim\co_1\ge 4$, but $d_{\sf max}(\co_2)$ is considerably less than $\frac{1}{2}\dim\co_2$, one can still detect such a phenomenon, see (7.1.1).
\\ \indent
In the examples below, $\co_2$ is lonely so that $d_{\sf max}(\co_2)=d_{\sf min}(\co_2)$ is known.

(7.1.1) $\g=\mathfrak{so}_8$. Here the intermediate orbit $\co_2=\co_{imd}$ corresponds to the partition
$(3,2,2,1)$, $\dim\co_2=16$, and $d_{\sf max}(\co_2)=5$. The boundary 
$\ov{\co_2}\setminus \co_2$ has three irreducible components of codimension $4$. The dense orbits 
in these components are even and they correspond to the partitions $(3,1^5)$ and $(2^4)$. 
(There are two different orbits associated with the {\sl very even\/} partition $(2^4)$.) If $\co_1$ is 
any of these three, then $\dim\co_1=12$ and $d_{\sf max}(\co_1)=12/2=6$. Hence 
$6=d_{\sf max}(\co_1)\le d_{\co_2} \le \frac{1}{2}\dim\co_2-1=7$. Then, by Proposition~\ref{d_O-max},
one actually obtains $d_{\co_2}=\displaystyle\max_{\co'\subset \ov{\co_2}}d_{\sf max}(\co')=6$ 
(it cannot be equal `7', since $\co_2$ is the only $G$-orbit of dimension $\ge 14$ in $\ov{\co_2}$ and 
$d_{\sf max}(\co_2)=5$.).

(7.1.2)  $\g=\GR{E}{n}$ ($n=6,7,8$) and $\co_2=\co_{imd}$. The information on $\co_2$ is presented in 
Table~\ref{table:except}. The boundary $\ov{\co_2}\setminus \co_2$ is irreducible and of 
codimension $2$, and we take $\co_1$ to be the dense orbit in the boundary. Then $\co_1$ appears to be Richardson, and we present its relevant data in Table~\ref{table:except2}.
Since $\co_2$ is not Richardson, we know that
$d_{\sf max}(\co_1)\le d_{\co_2} \le \frac{1}{2}\dim\co_2-1=d_{\sf max}(\co_1)$. Therefore, we also 
obtain the precise value of $d_{\co_2}=d_{\sf max}(\co_1)$.
\end{ex} 
\begin{table}[htb]  
\begin{center}
\begin{tabular}{c|cccc||cc}   
Algebra & $\co_1$ & {\sf wDd} of $\co_1$ & $\dim\co_1$  & 
$d_{\sf max}(\co_1)$ & $d_{\sf max}(\co_2)$ & $d_{\co_2}$ \\ \hline 
$\GR{E}{6}$ & $\GR{A}{4}+\GR{A}{1}$ & \begin{E6}{1}{1}{0}{1}{1}{1}\end{E6} 
 & 62   & 31 & 29 & 31\\
$\GR{E}{7}$ & $\GR{E}{7}(a_4)$ & \begin{E7}{2}{0}{0}{2}{0}{2}{0}\end{E7}
& 116  & 58 & 56 & 58 \\
$\GR{E}{8}$ & $\GR{E}{8}(b_4)$ & \begin{E8}{2}{2}{0}{0}{2}{0}{2}{0}\end{E8}
& 230  & 115 & 113 & 115 \\  \hline
\end{tabular} \vskip1ex
\caption{The Richardson  orbits $\co_1$ in $\GR{E}{n}$ related to Example~\ref{ex:anomaly}.2}
\label{table:except2}
\end{center}
\end{table}

(7.1.3) $\g=\sone$ ($n\ge 5$) and $\co_2=\co_{imd}$. The corresponding partition is
$(2n-5,2,2,1)$ and $d_{\sf max}(\co_2)=d_{\sf Dy}(\co_2)=n^2-n-7$. For $n\ge 5$, 
the dense orbits in the (reducible) boundary $\ov{\co_2}\setminus \co_2$ are
$\co'_1=\co(2n-5,1^5)$ and $\co''_1=\co(2n-7,3,3,1)$. Both these orbits are even and
their codimension equals 4 and 2, respectively.
The situation with $\co'_1$  (resp. $\co''_1$)   
is similar to that in (7.1.1) (resp. (7.1.2)). In particular,
$d_{\sf max}({\co''_1})\le d_{\co_2}\le \frac{1}{2}\dim\co_2-1=d_{\sf max}(\co''_1)$.
Hence
\[
   n^2-n-5=d_{\sf max}(\co''_1)=d_{\co_2}=\frac{1}{2}\dim\co_2-1 .
\]
\begin{rmk}
{\sf (i)} \ It is curious that, for all Richardson orbits of codimension $2$ in Example~\ref{ex:anomaly}
(i.e., $\co_1$ 
in Table~\ref{table:except2} and $\co''_1$ for $\sone$ with $n\ge 5$), the semisimple 
part of a polarisation, $[\el,\el]$,  is always of type $\GR{A}{2}+2\GR{A}{1}$.

{\sf (ii)} \ In the above examples of orbits $\co_2=\co_{imd}$ for $\GR{E}{n}$ and $\sone$ ($n\ge 4$), we know 
the exact (different) values of $d_{\sf max}(\co_2)$ and $d_{\co_2}$. But the the intermediate number 
$\bar d_{\co_2}$ is unknown! 
\end{rmk}

{\it \bfseries Some questions/open problems:}
\begin{enumerate}
\item The posets $\cI(\be)_\co$, $\co\in\fN/G$, may have many maximal and minimal elements.
However, in all known to us examples, the Hasse diagram of $\cI(\be)_\co$ appears to be connected.
Is it always the case?
\item Is it true that all minimal elements of the  poset $\cI(\be)_\co$ are of dimension $d_{\sf min}(\co)$?
\item Is it true that $\bar d_{\co}=d_{\co}$ for all $\co\in\fN/G$?
\end{enumerate}

{\bf Acknowledgements}. The authors are grateful to the anonymous referee for a useful suggestion.



\begin{thebibliography}{Pa95}

\bibitem{cp1} {\sc P.~Cellini} and {\sc P.~Papi}.
ad-nilpotent ideals of a Borel subalgebra, {\it J. Algebra},  {\bf 225}\,(2000),  130--141.

\bibitem{cp2} {\sc P.~Cellini} and {\sc P.~Papi}.
ad-nilpotent ideals of a Borel subalgebra II, {\it J. Algebra}, {\bf 258}\,(2002), 112--121.

\bibitem{cm} {\sc D.~Collingwood} and {W.~McGovern}. ``{\it Nilpotent orbits in
semisimple Lie algebras}", Mathematics Series, Van Nostrand Reinhold, 1993.

\bibitem{ela76} {\sc A.G.\,Elashvili}.  The centralizers of nilpotent elements in 
semisimple Lie algebras, {\it Trudy Razmadze Matem. Inst.}~(Tbilisi),   
{\bf 46}\,(1975), 109--132 (in Russian). (\href{http://www.ams.org/mathscinet-getitem?mr=0393148}{MR0393148})

\bibitem{fang14}
{\sc C.~Fang}. Ad-nilpotent ideals of minimal dimension,
{\it J. Algebra}, {\bf 403}\,(2014), 517--543.

\bibitem{nil-4}
{\sc M.~Gerstenhaber}. 
On nilalgebras and linear varieties of nilpotent matrices, IV. 
{\it Ann. of Math.}, (2) {\bf 75}\,(1962), 382--418.

\bibitem{jos84}
{\sc A.~Joseph}.  On the variety of a highest weight module,
{\it J. Algebra}, {\bf 88}\,(1984), 238--278.

\bibitem{kaw86}
{\sc N.~Kawanaka}.  Generalized Gelfand-Graev representations of exceptional simple
groups over a finite field I, {\it Invent. Math.}, {\bf 84}\,(1986), 575--616.

\bibitem{ke83} 
{\sc G.~Kempken}. 
Induced conjugacy classes in classical Lie-algebras,
{\it Abh. Math. Sem. Univ. Hamburg}, {\bf 53}\,(1983), 53--83.

\bibitem{manuscr} {\sc D.~Panyushev}.
Complexity and nilpotent orbits, {\it Manuscripta Math.},
{\bf 83}\,(1994), 223--237.

\bibitem{pr2}  {\sc D.~Panyushev} and {\sc G.~R\"ohrle}.
On spherical ideals of Borel subalgebras,  
{\it Archiv Math.}, {\bf 84}\,(2005), 225--232.

\bibitem{som}  {\sc E.~Sommers}.
Equivalence classes of ideals in the nilradical of a Borel subalgebra, 
{\it Nagoya Math. J.}, {\bf 183}\,(2006), 161--185. 

\bibitem{spalt77} {\sc N.~Spaltenstein}.
On the fixed point set of a unipotent element on the variety of Borel subgroups, 
{\it Topology}, {\bf 16}\,(1977),  203--204.

\bibitem{spalt82} {\sc N.~Spaltenstein}. 
{\it ``Classes Unipotentes et Sous-groupes de Borel"}, Lecture notes in Math. {\bf 946},  Berlin 
Heidelberg New York: Springer 1982.

\bibitem{St76} {\sc R.~Steinberg}. 
On the desingularization of the unipotent variety, {\it Invent. Math.}, {\bf 36}\,(1976), 209--224.

\bibitem{t41} {\rusc {E1}.B.~Vinberg, V.V.~Gorbatseviq, A.L.~Oniwik}.
``{\rusi Gruppy i algebry Li--3}'', {\rus  Sovrem. probl. matematiki. Fundam. napravl., t.\,41.
Moskva: VINITI} 1990 (Russian). 
English translation: {\sc V.V.\,Gorbatsevich, A.L.\,Onishchik} and {\sc E.B.\,Vinberg}.
``Lie Groups and Lie Algebras III'' (Encyclopaedia Math. Sci., vol.~41) Berlin: Springer 1994.

\end{thebibliography}
\end{document}